\documentclass{amsart}

\usepackage[english]{babel}
\usepackage{latexsym,amsfonts,amsmath,amssymb,mathrsfs}
\usepackage[pdftex]{graphicx}
\usepackage{amsthm}
\usepackage{enumerate}
\usepackage{bbm}
\usepackage{amsthm}
\usepackage[pdfborder={0 0 0}]{hyperref}
\usepackage{comment}

\usepackage[textsize=footnotesize,color=yellow, bordercolor=white]{todonotes}

\newcommand{\comm}[1]{}

\DeclareRobustCommand{\SkipTocEntry}[5]{}

\newtheorem{satz}{Satz}[section]
\newtheorem{thm}[satz]{Theorem}
\newtheorem{prop}[satz]{Proposition}

\newcounter{cl}[satz]
\newtheorem{claim}[cl]{Claim}
\newcounter{subcl}[cl]
\newtheorem{cor}[satz]{Corollary}

\theoremstyle{definition}
\newtheorem{question}{Question}
\numberwithin{subsection}{section}

\newcounter{anhang}
\newcommand{\thistheoremname}{}
\newtheorem*{genericthm}{\thistheoremname}

\newcommand{\st}{\mid}

\newcommand{\ZFC}{\ensuremath{\operatorname{ZFC}} }

\newcommand{\GCH}{\ensuremath{\operatorname{GCH}} }

\newcommand{\Ord}{\ensuremath{\operatorname{Ord}} }
\newcommand{\HOD}{\ensuremath{\operatorname{HOD}} }

\newcommand{\Card}{\ensuremath{\operatorname{Card}}}

\newcommand{\forces}{\Vdash}
\newcommand{\lt}{{\smalllt}}
\newcommand{\smalllt}{\mathrel{\mathchoice{\raise2pt\hbox{$\scriptstyle<$}}{\raise1pt\hbox{$\scriptstyle<$}}{\raise0pt\hbox{$\scriptscriptstyle<$}}{\scriptscriptstyle<}}}
\newcommand{\smallleq}{\mathrel{\mathchoice{\raise2pt\hbox{$\scriptstyle\leq$}}{\raise1pt\hbox{$\scriptstyle\leq$}}{\raise1pt\hbox{$\scriptscriptstyle\leq$}}{\scriptscriptstyle\leq}}}

\newcommand{\pwimg}{\, "}

\newcommand{\cf}{\ensuremath{\operatorname{cf}}}

\newcommand{\dom}{\ensuremath{\operatorname{dom}}}

\newcommand{\V}{\mathscr V}

\newcommand{\cM}{\mathcal{M}}
\newcommand{\cN}{\mathcal{N}}

\newcommand{\cF}{\mathcal{F}}

\newcommand{\image}{\mathbin{\hbox{\tt\char'42}}}
\newcommand{\restrict}{\upharpoonright}

\newcommand{\p}{\mathbb P}
\newcommand{\la}{\langle}
\newcommand{\ra}{\rangle}
\newcommand{\q}{\mathbb Q}
\newcommand{\Coll}{\ensuremath{\operatorname{Coll}}}
\newcommand{\tail}{\text{tail}}
\newcommand{\sm}{\text{small}}

\newcommand{\Add}{\ensuremath{\operatorname{Add}}}
\newcommand{\Levy}{L\'{e}vy}
\newcommand{\Godel}{G\"odel}
\newcommand{\one}{\mathbbm{1}} % requires \usepackage{bbm}

\newcommand{\Hull}{\ensuremath{\operatorname{Hull}}}

\theoremstyle{definition}

\theoremstyle{remark}

\theoremstyle{definition}
\newtheorem{case}{Case}
\author{Sy-David Friedman} \address[S.-D. Friedman] {Kurt \Godel\
  Research Center, Institut f\"ur Mathematik, UZA 1, Universit\"at
  Wien. Augasse 2-6, 1090 Wien, Austria}
\email{sdf@logic.univie.ac.at}
\urladdr{http://www.logic.univie.ac.at/~sdf/} \thanks{The first-listed
  author wishes to acknowledge the support of the Austrian Science
  Fund (FWF) through Research Projects P28157 and P28420.}
\author{Victoria Gitman} \address[V. Gitman]{The City University of
  New York, CUNY Graduate Center, Mathematics Program, 365 Fifth
  Avenue, New York, NY 10016} \email{vgitman@nylogic.org}
\urladdr{https://victoriagitman.github.io/} \author{Sandra M\"uller}
\address[S. M\"uller] {Kurt \Godel\ Research Center, Institut f\"ur
  Mathematik, UZA 1, Universit\"at Wien. Augasse 2-6, 1090 Wien,
  Austria} \email{mueller.sandra@univie.ac.at}
\urladdr{https://muellersandra.github.io/} \thanks{The third-listed
  author, formerly known as Sandra Uhlenbrock, wishes to acknowledge
  the support of the Austrian Science Fund (FWF) through Research
  Project P28157. We would like to thank Gabriel Goldberg for his valuable comments on an earlier draft of the article.}

\title{Structural Properties of the Stable Core}

\date{\today}

\begin{document}
\maketitle

\begin{abstract}
The stable core, an inner model of the form $\la L[S],\in, S\ra$ for a simply definable predicate $S$, was introduced by the first author in \cite{Friedman:StableCore}, where he showed that $V$ is a class forcing extension of its stable core. We study the structural properties of the stable core and its interactions with large cardinals. We show that the $\GCH$ can fail at all regular cardinals in the stable core, that the stable core can have a discrete proper class of measurable cardinals, but that measurable cardinals need not be downward absolute to the stable core. Moreover, we show that, if large cardinals exist in $V$, then the stable
core has inner models with a proper class of measurable limits of measurables, with a proper class of measurable limits of measurable limits of measurables, and so forth. We show this by providing a characterization of natural inner models $L[C_1, \dots, C_n]$ for specially nested class clubs
$C_1, \dots, C_n$, like those arising in the stable core, generalizing recent results of Welch \cite{Welch:ClubClassHartigQuantifierModel}.
\end{abstract}

\section{Introduction}
The first author introduced the inner model \emph{stable core} while investigating under what
circumstances the universe $V$ is a class forcing extension of the
inner model $\HOD$, the collection of all hereditarily ordinal
definable sets \cite{Friedman:StableCore,Friedman:EnrichedStableCore}. He showed in \cite{Friedman:StableCore} that there is a robust $\Delta_2$-definable
class $S$ contained in $\HOD$ such that $V$ is a class-forcing extension of the structure \hbox{$\la L[S],\in, S\ra$}, which he called the stable core, by an $\Ord$-cc class partial order $\p$ definable from $S$. Indeed, for any inner model $M$, $V$ is a $\p$-forcing extension of $\la M[S],\in,S\ra$, so that in particular, since $\HOD[S]=\HOD$, $V$ is a $\p$-forcing extension of $\la \HOD,\in,S\ra$.

Let's explain the result in more detail for the stable core $L[S]$, noting that exactly the same analysis applies to $\HOD$. The partial order $\p$ is definable in \hbox{$\la L[S],\in,S\ra$} and there is a generic filter $G$, meeting all dense sub-classes of $\p$ definable in \hbox{$\la L[S],\in,S\ra$},
such that $V=L[S][G]$. All standard forcing theorems hold for $\p$ since it has the $\Ord$-cc. Thus, we get that the forcing relation for $\p$ is definable in $\la L[S],\in,S\ra$ and the forcing extension $\la V,\in,G\ra\models\ZFC$. However, this particular generic filter $G$ is not definable in $V$. To obtain $G$, we first force with an auxiliary forcing $\q$ to add a particular class $F$, without adding sets, such that $V=L[F]$. We then show that $G$ is definable from $F$ and $F$ is in turn definable in the structure $\la L[S][G], \in, S,G\ra$, so that $L[S][G]=V$. This gives a formulation of the result as a $\ZFC$-theorem because we can say (using the definitions of $\p$ and $\q$) that it is forced by $\q$ that $V=L[F]$, where $F$ is $V$-generic for $\q$, and (the definition of) $G$ is $\la L[S],\in,S\ra$-generic, and finally that $F$ is definable in $\la L[S][G],\in,S,G\ra$. Of course, a careful formulation would say that the result holds for all sufficiently large natural numbers $n$, where $n$ bounds the complexity of the formulas used.

Without the niceness requirement on $\p$ that it has the $\Ord$-cc, there is a much easier construction of a class forcing notion $\p$, suggested by Woodin, such that $V$ is a class forcing extension of $\la \HOD,\in, \p\ra$ (see the end of Section~\ref{sec:preliminaries}). At the same time, some additional predicate must be added to $\HOD$ in
order to realize all of $V$ as a class-forcing extension because, as
Hamkins and Reitz observed in
\cite{HamkinsReitz:The-set-theoretic-universe-is-not-necessarily-a-forcing-extension-of-HOD},
it is consistent that $V$ is not a class-forcing extension of $\HOD$. To construct such a counterexample, we suppose that $\kappa$ is inaccessible in $L$ and force over the Kelley-Morse model $\mathcal L=\la L_\kappa,\in, L_{\kappa+1}\ra$ to code the truth predicate of $L_\kappa$ (which is an element of $L_{\kappa+1}$) into the continuum pattern below $\kappa$. The first-order part $L_\kappa[G]$ of this extension cannot be a forcing extension of $\HOD^{L_\kappa[G]}=L_\kappa$ (by the weak homogeneity of the coding forcing), because the truth predicate of $L_\kappa$ is definable there and this can be recovered via the forcing relation.

While the definition of the partial order $\p$ is fairly involved, the
\emph{stability predicate} $S$ simply codes the elementarity relations
between sufficiently nice initial segments $H_\alpha$ (the collection of all sets with transitive closure of size less than $\alpha$) of $V$. Given a
natural number $n\geq 1$, call a cardinal $\alpha$ $n$-\emph{good} if
it is a strong limit cardinal and $H_\alpha$ satisfies
$\Sigma_n$-collection. The predicate $S$ consists of
triples $(n,\alpha,\beta)$ such that $n\geq 1$, $\alpha$ and $\beta$
are $n$-good cardinals and $H_\alpha\prec_{\Sigma_n}H_\beta$. We will denote by $S_n$ the $n$-th slice of the stability predicate $S$, namely $S_n=\{(\alpha,\beta)\mid (n,\alpha,\beta)\in S\}$.\footnote{Here we simplify the definition of $S$ originally given in \cite{Friedman:StableCore} to make it easier to work with. We do not claim that the definitions are equivalent, only that, it is not difficult to check that all the results from \cite{Friedman:StableCore} still hold with the definition given here.}

Clearly the stable core $L[S]\subseteq\HOD$, and the first author showed in \cite{Friedman:StableCore} that
it is consistent that $L[S]$ is smaller than $\HOD$. The stable core has several nice properties which fail for $\HOD$ such as that it is partially forcing absolute and, assuming $\GCH$, is preserved by forcing to code the universe into a real \cite{Friedman:StableCore}.

In order to
motivate the many questions which arise about the stable core let us
briefly discuss the set-theoretic goals of studying inner models.

The study of canonical inner models has proved to be one of the most
fruitful directions of modern set-theoretic research. The canonical
inner models, of which \Godel's constructible universe $L$ was the first
example, are built bottom-up by a canonical procedure. The resulting
fine structure of the models leads to regularity properties, such as
the $\GCH$ and $\square$, and sometimes even absoluteness
properties. But all known canonical inner models are incompatible with
sufficiently large large cardinals, and indeed each such inner model
is very far from the universe in the presence of sufficiently large
large cardinals in the sense, for example, that covering fails and the
large cardinals are not downward absolute.

The inner model $\HOD$ was introduced by \Godel, who showed that in a
universe of ${\rm ZF}$ it is always a model of $\ZFC$. But unlike the
constructible universe which also shares this property, $\HOD$ has
turned out to be highly non-canonical. While $L$ cannot be modified by
forcing, $\HOD$ can be easily changed by forcing because we can use
forcing to code information into $\HOD$. For instance, any subset of
the ordinals from $V$ can be made ordinal definable in a set-forcing extension
by coding its characteristic function into the continuum pattern, so
that it becomes an element of the $\HOD$ of the extension. Indeed, by
coding all of $V$ into the continuum pattern of a class-forcing
extension, Roguski showed that every universe $V$ is the $\HOD$ of one
of its class-forcing extensions \cite{Roguski:HOD}. Thus, any
consistent set-theoretic property, including all known large cardinals, consistently holds in $\HOD$. At the same time, the
$\HOD$ of a given universe can be very far from it. It is consistent
that every measurable cardinal is not even weakly compact in $\HOD$
and that a universe can have a supercompact cardinal which is not even
weakly compact in $\HOD$
\cite{ChengFriedmanHamkins:LargeCardinalsNeedNotBeLargeInHOD}. It is
also consistent that $\HOD$ is wrong about all successor cardinals
\cite{CummingsFriedmanGolshani:CollapsingCardinalsHOD}.

Does the stable core behave more like the canonical inner models or
more like $\HOD$? Is there a fine structural version of the stable
core, does it satisfy regularity properties such as the $\GCH$? Is
there a bound on the large cardinals that are compatible with the stable core? Or, on the other hand, are the large cardinals
downward absolute to the stable core? Can we code information into the
stable core using forcing?

In this article, we show the following results about the structure of
the stable core, which answer some of the aforementioned questions as
well as motivate further questions about the structure of the stable
core in the presence of sufficiently large large cardinals.

Measurable cardinals are consistent with the stable core.
\begin{thm}$\,$
\begin{itemize}
\item[(1)] The stable core of $L[\mu]$, the canonical model for one measurable
  cardinal, is $L[\mu]$. In particular, the stable core can have a measurable cardinal.
\item[(2)] If $\vec U=\la U_\alpha\mid\alpha\in\Ord\ra$ is a discrete collection of normal measures, then the stable core of $L[\vec U]$ is $L[\vec U]$. In particular, the stable core can have a discrete proper class of measurable cardinals.
\end{itemize}
\end{thm}

We can code information into the stable core over $L$ or $L[\mu]$
using forcing.
\begin{thm}
  Suppose $\p\in L$ is a forcing notion and $G\subseteq \p$ is
  $L$-generic. Then there is a further forcing extension $L[G][H]$
  such that $G\in L[S^{L[G][H]}]$ (the universe of the stable core). An analogous result holds for
  $L[\mu]$.
\end{thm}
An extension of the coding results shows that the $\GCH$ can fail
badly in the stable core.
\begin{thm}$\,$
\begin{itemize}
\item[(1)] There is a class-forcing extension of $L$ such that in its stable core the $\GCH$
  fails at every regular cardinal.
\item[(2)] There is a class-forcing extension of $L[\mu]$ such that in its stable core there is
  a measurable cardinal and the $\GCH$ fails on a tail of regular
  cardinals.
\end{itemize}
\end{thm}
Measurable cardinals need not be downward absolute to the stable core.
\begin{thm}
  There is a forcing extension of $L[\mu]$ in which the measurable
  cardinal $\kappa$ of $L[\mu]$ remains measurable, but it is not even
  weakly compact in the stable core.
\end{thm}

Although we don't know whether the stable core can have a measurable limit of measurables, the stable core has inner models with measurable limits of measurables, and much more. Say that a cardinal $\kappa$ is $1$-\emph{measurable}
if it is measurable, and, for $n < \omega$, $(n+1)$-\emph{measurable} if it
is measurable and a limit of $n$-measurable cardinals. Write $m_0^\#$ for $0^\#$ and $m_n^\#$ for the minimal mouse which is a sharp for a proper class of $n$-measurable cardinals, namely, an active mouse $\cM$ such that the critical point of the top extender is a limit of $n$-measurable cardinals in $\cM$. Here we mean \emph{mouse} in the sense of \cite[Sections 1 and 2]{Mitchell:BeginningInnerModelTheory}, i.e., a mouse has only total measures on its sequence. The mouse $m_n^{\#}$ can also be construed as a fine structural
mouse with both total and partial extenders (see
\cite{Zeman:InnerModelsLargeCardinals}, Section 4).
\begin{thm}
  For all $n<\omega$, if $m_{n+1}^\#$ exists, then $m_n^\#$ is in the
  stable core.
\end{thm}

Moreover, we obtain the following characterization of natural inner
models of the stable core.

\begin{thm}
  Let $n < \omega$ and suppose that $m_n^\#$ exists. Then whenever
  $$C_1\supseteq C_2\supseteq\ldots\supseteq C_n$$ are class clubs of
  uncountable cardinals such that for every $1<i\leq n$ and every
  $\gamma \in C_i$,
\[\la H_\gamma,\in, C_1,\ldots,C_{i-1}\ra\prec_{\Sigma_1} \la V,\in,
  C_1,\ldots,C_{i-1}\ra,\] then $L[C_1,\ldots,C_n]$ is a hyperclass-forcing extension of a (truncated) iterate of $m_n^\#$.
\end{thm}

\section{Preliminaries}\label{sec:preliminaries}
Recall that, for a cardinal $\alpha$, $H_\alpha$ is the collection of
all sets $x$ with transitive closure of size less than $\alpha$. If $\alpha$ is regular, then
$H_\alpha$ satisfies $\ZFC^-$ ($\ZFC$ without the powerset axiom). But
for singular $\alpha$, $H_\alpha$ may fail to satisfy even $\Sigma_2$-collection.

The following proposition is standard.
\begin{prop}\label{prop:hereditaryElementarity} Suppose $\alpha$ and $\beta$ are infinite cardinals.
\begin{itemize}
\item[(1)] $H_\alpha\prec_{\Sigma_1} V$.
\item[(2)] If $H_\alpha\prec_{\Sigma_m} V$, then $\Sigma_m$-collection holds in $H_\alpha$. In particular, every $H_\alpha$ satisfies $\Sigma_1$-collection.
\item[(3)] If $H_\alpha\prec_{\Sigma_m} H_\beta$ and $\Sigma_m$-collection holds in $H_\beta$, then it also holds in $H_\alpha$.
\end{itemize}
\end{prop}
\begin{proof}
Let's prove (1), which is a classical fact attributed to \Levy. Suppose $\exists x\,\varphi(x,a)$ holds in $V$, where $\varphi(x,a)$ is a $\Delta_0$-formula and $a\in H_\alpha$. We can assume without loss that $a$ is transitive. Let $X\prec_{\Sigma_1} V$ be a $\Sigma_1$-elementary substructure of size $|a|$ with $a\cup \{a\}\subseteq X$, and let $M$ be the Mostowski collapse of $X$. Since $M$ is transitive and has size $|a|$, it is in $H_\alpha$. Also, by elementarity, $M$ satisfies $\exists x\,\varphi(x,a)$. So there is $b\in M$ such that $M\models\varphi(b,a)$. But since $M\subseteq H_\alpha$ is transitive and $\varphi(x,y)$ is a $\Delta_0$-assertion, it follows that $H_\alpha$ satisfies $\varphi(b,a)$ as well.

To prove $(2)$, suppose that $\varphi(x,y)$ is a $\Sigma_m$-formula and fix $a\in H_\alpha$. In $V$, by collection, there is a set $b$ such that $\forall x\in a\,\exists y\in b\, \varphi(x,y)$. So $V$ satisfies $$\exists z\,\forall x\in a\,\exists y\in z\,\varphi(x,y),$$ which is canonically equivalent in $V$ to a $\Sigma_m$-assertion $\psi(a)$. By the available elementarity, $\psi(a)$ holds in $H_\alpha$, but $\psi(a)$ clearly implies $\exists z\,\forall x\in a\,\exists y\in z\,\varphi(x,y)$ even in the absence of collection, and so $H_\alpha$ has a collecting set for $a$ and $\varphi$. An analogous argument shows $(3)$.
\end{proof}
It follows immediately from Proposition~\ref{prop:hereditaryElementarity}~(1) that the strong limit cardinals of $V$ are definable in the stable core.
\begin{cor}\label{cor:stronglimDefinable}
The class of strong limit cardinals of $V$ is definable in the stable core $\la L[S],\in,S\ra$. Indeed, $\alpha$ is a strong limit cardinal if and only if there is a cardinal $\beta$ such that $(\alpha,\beta)\in S_1$.
\end{cor}
The stable core can also define, for each $n$, the class club $C_n$ of all strong limit cardinals $\alpha$ such that $H_\alpha\prec_{\Sigma_n}V$.
\begin{prop}\label{prop:SigmaElementaryClubs}
For every $n<\omega$, $C_n$ is definable in the stable core.
\end{prop}
\begin{proof}
The class club $C_1$ is definable because it is precisely the class of all strong limit cardinals. Let's argue that $C_2$ is precisely the collection of all $\alpha$ such that for cofinally many $\beta$, we have $\la \alpha,\beta\ra\in S_2$. Suppose that $\alpha$ is a strong limit cardinal such that $H_\alpha\prec_{\Sigma_2} V$. Then whenever $\beta$ is a strong limit cardinal and $H_\beta\prec_{\Sigma_2}V$, we have $\la \alpha,\beta\ra\in S_2$ since both satisfy $\Sigma_2$-collection by Proposition~\ref{prop:hereditaryElementarity}~(2). Now suppose that $\alpha$ is a strong limit cardinal such that for cofinally many $\beta$, $H_\alpha\prec_{\Sigma_2}H_\beta$. We will argue that $H_\alpha\prec_{\Sigma_2}V$. Suppose $H_\alpha\models\forall x\,\exists
y\,\varphi(x,y,a)$, where $\varphi$ is $\Delta_0$ and $a \in H_\alpha$. Given a set $b$, by assumption, there is a $\beta$ with $b\in H_\beta$ such that $H_\alpha\prec_{\Sigma_2}H_\beta$. It follows that $H_\beta\models\exists y\varphi(b,y,a)$, and hence so does $V$. A similar argument now shows that each $C_n$ is the collection of all $\alpha$ such that for cofinally many $\beta$, $\la\alpha,\beta\ra\in S_n$.
\end{proof}

Given a cardinal $\alpha$, let $H^{\lt\alpha}$ denote the relation consisting of pairs $\la \beta,H_\beta\ra$ for $\beta<\alpha$.
\begin{prop}\label{prop:reducingElementarity}
  For $m\geq 1$ and strong limit cardinals $\alpha$ and $\beta$, $H_\alpha\prec_{\Sigma_{m+1}}H_\beta$ if and only if
  $\la H_\alpha,\in,H^{\lt\alpha}\ra\prec_{\Sigma_m} \la
  H_\beta,\in,H^{\lt\beta}\ra$.
\end{prop}
\begin{proof}

For the forward direction, observe that the relation
  $H^{\lt\alpha}$ is $\Pi_1$-definable and amenable over $H_\alpha$, which implies that predicates which are $\Sigma_m$-definable over $\la H_\alpha,\in,H^{\lt\alpha}\ra$ are $\Sigma_{m+1}$-definable over $H_\alpha$. So let's focus on the
  backward direction. First, observe that a $\Sigma_2$-formula
  $\exists x\,\forall y\,\varphi(x,y,a)$ holds in $H_\alpha$ if and
  only if the $\Sigma_1$-formula
  $$\exists z\,[z=(\beta,H_\beta)\wedge \exists x\in H_\beta\,\forall
  y\in H_\beta\,\varphi(x,y,a)]$$ holds in
  $\la H_\alpha,\in,H^{\lt\alpha}\ra$, and a $\Pi_2$-formula
  $\forall x\,\exists y\,\varphi(x,y,a)$ holds in $H_\alpha$ if and
  only if the $\Pi_1$-formula
  $$\forall z[z=(\beta,H_\beta)\rightarrow \forall x\in H_\beta\,\exists
  y\in H_\beta\,\varphi(x,y,a)]$$ holds in
  $\la H_\alpha,\in,H^{\lt\alpha}\ra$. Both equivalences follow
  from Proposition~\ref{prop:hereditaryElementarity} (1) and the fact that $\alpha$ and $\beta$ are strong limits. Thus, the
  complexity of any assertion is reduced by 1.
\end{proof}
\begin{prop}\label{prop:liftingElementarity}
Suppose $1\leq m<\omega$, $\alpha$ and $\beta$ are strong limit cardinals, $\p\in H_\alpha$ is a partial order, and $G\subseteq \p$ is $V$-generic. For $(1)$ and $(2)$, suppose additionally that $H_\alpha\models\Sigma_m$-collection.
\begin{itemize}
\item[(1)] The Definability Lemma and Truth Lemma for $\Sigma_m$-formulas hold for $\p$ in $H_\alpha$. Indeed, if $\varphi(\bar x)$ is a $\Sigma_m$-formula, then the relation $p\forces\varphi(\bar x)$ is also $\Sigma_m$ in $H_\alpha$.
\item[(2)] $H_\alpha\prec_{\Sigma_m} H_\beta$ if and only if $$H_\alpha^{V[G]}=H_\alpha[G]\prec_{\Sigma_m}H_\beta[G]=H_\beta^{V[G]}.$$
\item[(3)] $H_\alpha$ satisfies $\Sigma_m$-collection if and only if $H_\alpha[G]$ satisfies $\Sigma_m$-collection.
\end{itemize}
\end{prop}
\begin{proof}
  The argument for (1) actually works for all cardinals $\alpha$ and
  $\beta$, not just strong limits. We argue that the standard definition of the forcing
  relation works in $H_\alpha$. Suppose, for instance, that $H_\alpha$
  satisfies $p\forces\sigma=\tau$ for $\p$-names $\sigma,\tau \in
  H_\alpha$ and let $H\subseteq \p$ be
  $V$-generic with $p\in H$. The relation $p\forces\sigma=\tau$ is a
  $\Sigma_1$-assertion stating that a tree exists witnessing the
  recursive definition of $\sigma=\tau$ in terms of names of lower
  rank (in fact, the assertion is $\Delta_1$ because we can say ``for
  every tree obeying the recursive definition..."). So by
  $\Sigma_1$-elementarity, $p\forces\sigma=\tau$ holds in $V$, and
  hence $\sigma_H=\tau_H$. Conversely, suppose that $\sigma_H=\tau_H$
  for some $V$-generic filter $H\subseteq \p$. Then there is $p\in H$ such
  that $p\forces \sigma=\tau$, and hence, by $\Sigma_1$-elementarity,
  $p\forces\sigma=\tau$ holds in $H_\alpha$ as well. The remainder of
  the argument is by induction on the complexity of formulas. For
  instance, let's argue for negations. Suppose that the standard
  definition of the forcing relation holds in $H_\alpha$ for a formula
  $\varphi$. By definition of the forcing relation,
  $p\forces\neg\varphi$ if for every $q\leq p$, $q$ does not force
  $\varphi$, but clearly this holds in $H_\alpha$ if and only if it
  holds $V$ provided that they agree on what it means for $q$ to force
  $\varphi$, which is the inductive assumption.

  The argument that the definition of the forcing relation for a
  $\Sigma_m$-formula is itself $\Sigma_m$ is also
  standard. The collection assumption is required to make sure that a formula is equivalent to its normal form where all the bounded quantifiers are pushed to the back. The argument above already shows that for formulae of the
  form ``$\sigma=\tau$'' the forcing relation is $\Delta_1$. Let's
  argue for instance that for $\Delta_0$-formulas, the complexity of
  the forcing relation is $\Delta_1$. Say
  $p\forces \exists x\in\sigma\,\varphi(x,\sigma)$, where
  $\varphi(x,y)$ is a $\Delta_0$-formula and by induction
  $q\forces\varphi(x,y)$ is a $\Delta_1$-relation. Then
  $p\forces\exists x\in\sigma\,\varphi(x,\sigma)$ holds if and only if
  for every $q\leq p$, there is $r\leq q$ and
  $\tau\in \text{dom}(\sigma)$ such that
  $r\forces\varphi(\tau,\sigma)$, and of course, quantification over
  elements of $\p$ is obviously bounded.

Now let's prove (2). We start with the forward direction, which is standard. Suppose that \hbox{$H_\alpha\prec_{\Sigma_m} H_\beta$}. Clearly, since $\p\in H_\alpha$, we have $H_\alpha[G]=H_\alpha^{V[G]}$ and similarly for $H_\beta$. If a $\Sigma_m$-assertion $\varphi$ holds in $H_\alpha[G]$, then there is some $p\in G$ such that $p\forces\varphi$ holds in $H_\alpha$, which is also a $\Sigma_m$-assertion by (1), and so $p\forces\varphi$ holds in $H_\beta$, meaning that $H_\beta[G]$ satisfies $\varphi$.

Next, let's prove the backward direction. Suppose that
$H_\alpha[G]\prec_{\Sigma_m}H_\beta[G]$. The argument for $m=1$ is
trivial since if $\alpha$ and $\beta$ are cardinals in $V[G]$, then
they are also obviously cardinals in $V$, and so the result follows by
Proposition~\ref{prop:hereditaryElementarity} (1). So suppose that
$m\geq 2$. Since $\p\in H_\alpha$, $\alpha$ remains a strong limit in
$V[G]$. Thus, $H_\alpha[G]=H_\alpha^{V[G]}$ has a definable hierarchy consisting of
$H_\beta^{V[G]}$ for regular $\beta<\alpha$. The existence of such a
hierarchy suffices for the standard $\Delta_2$-definition of the
ground model in a forcing extension (due independently to
Woodin~\cite{woodin:groundmodel} and Laver~\cite{laver:groundmodel}) to
go through, so that $H_\alpha$ is $\Delta_2$-definable in
$H_\alpha[G]$. Indeed, examining the definition shows that
$\la H_\alpha,\in, H^{\lt\alpha}\ra$ is $\Delta_1$-definable in
\hbox{$\la H_\alpha[G],\in, (H^{\lt\alpha})^{V[G]}\ra$}. Now suppose that
$H_\alpha$ satisfies a $\Sigma_m$-assertion $\varphi(x,a)$, and let
$\varphi^*(x,a)$ be the equivalent $\Sigma_{m-1}$-assertion which
holds in $\la H_\alpha,\in, H^{\lt\alpha}\ra$. Since
$\la H_\alpha,\in,H^{\lt\alpha}\ra$ is $\Delta_1$-definable in
$\la H_\alpha[G],\in, (H^{\lt\alpha})^{V[G]}\ra$, there is a
$\Sigma_{m-1}$-assertion $\varphi^{**}(x,a)$ expressing in
$\la H_\alpha[G],\in, (H^{\lt\alpha})^{V[G]}\ra$ that $\varphi^*(x,a)$ holds in
$\la H_\alpha,\in, H^{\lt\alpha}\ra$. By
Proposition~\ref{prop:reducingElementarity},
$$\la H_\alpha[G],\in,(H^{\lt\alpha})^{V[G]}\ra\prec_{\Sigma_{m-1}}\la
H_\beta[G],\in,(H^{\lt\beta})^{V[G]}\ra.$$ Thus,
$\la H_\beta[G],\in, (H^{\lt\beta})^{V[G]}\ra$ satisfies
$\varphi^{**}(x,a)$, and therefore $\varphi^*(x,a)$ holds in
$\la H_\beta,\in, H^{\lt \beta}\ra$. So finally, $\varphi(x,a)$
holds in $H_\beta$.

Finally, let's prove (3). Again, we start with the standard forward direction. Suppose that $H_\alpha$ satisfies $\Sigma_m$-collection. Let $\varphi(x,y)$ and $a$ be such that $$H_\alpha[G]\models\forall x\in a\,\exists y\,\varphi(x,y).$$ So there is some $p\in G$ and a name $\dot a$ for $a$ such that $p\forces\forall x\in\dot a\,\exists y\,\varphi(x,y)$. Fix a name $\sigma\in\dom \dot a$ and apply $\Sigma_m$-collection in $H_\alpha$ to the statement $$\forall q\leq p\,\exists y\,(q\forces\sigma\in\dot a\rightarrow q\forces\varphi(\sigma,y))$$ to obtain a collecting set $y_\sigma$. Next, apply $\Sigma_m$-collection in $H_\alpha$, to the statement $$\forall x\in\dom \dot a\,\exists z\,\exists q\leq p\, (q\forces x\in\dot a\rightarrow \exists y\in z\,q\forces \varphi(x,y)),$$ which holds by the previous step because $y_x$ witnesses it for $x$, to obtain a collecting set $B$. We can assume without loss that $B$ consists only of $\p$-names and let $\dot b =\{(y,p)\mid y\in B\}$. It is not difficult to see that $\dot b_G$ gives the collecting set in $H_\alpha[G]$.

For the backward direction, assume that $H_\alpha[G]$ satisfies
$\Sigma_m$-collection and let $\varphi(x,y)$ and $a$ be such that
$H_\alpha$ satisfies $\forall x\in a\,\exists y\,\varphi(x,y)$. Again,
the case $m=1$ is trivial since cardinals are downward absolute, so we
can assume $m\geq 2$ and use the $\Delta_1$-definability of $\la
H_\alpha,\in,H^{\lt\alpha}\ra$ in $\la H_\alpha[G],\in,
(H^{\lt\alpha})^{V[G]}\ra$. Thus, we can apply $\Sigma_{m}$-collection in
$H_\alpha[G]$ to obtain a set $b$ collecting witnesses for
$\varphi(x,y)$. Since $\p$ can be assumed to have size less than $\alpha$, we can cover
$b\cap V$ with a set $\bar b$ of size less than $\alpha$ in $V$. So
$\bar b\in H_\alpha$.
\end{proof}
It follows from Proposition~\ref{prop:liftingElementarity}~(2)~and~(3) that only an initial segment of the stability predicate can be changed by set forcing. So the stable core is at least partially forcing absolute.
\begin{cor}\label{cor:forcingabsolutness}
  If $\p\in H_\gamma$ is a forcing notion and $G\subseteq \p$ is
  $V$-generic, then $(n,\alpha,\beta)\in S$ if and only if
  $(n,\alpha,\beta)\in S^{V[G]}$ for all $\alpha, \beta \geq
  \gamma$. So, in particular, $S$ and $S^{V[G]}$ agree above the size
  of the forcing.
\end{cor}

Next, let's give an argument that consistently the stable core can be a proper submodel of $\HOD$. The fact follows from results in \cite{Friedman:StableCore}, but here we give a simplified argument suggested to the second author by Woodin.
\begin{prop}
It is consistent that $L[S]\subsetneq\HOD$.
\end{prop}
\begin{proof}
  Start in $L$ and force to add a Cohen real $r$. Next, force to code
  $r$ into the continuum pattern on the $\aleph_n$'s and let $H$ be
  $L[r]$-generic for the coding forcing $\p$ (the full support
  $\omega$-length product forcing on coordinate $n$ with
  $\Add(\aleph_n,\aleph_{n+2})$ whenever $n\in r$ and with trivial
  forcing otherwise). Observe that $\HOD^{L[r][H]}=L[r]$ because it
  has $r$, which the forcing $\p$ made definable, and it must be
  contained in $L[r]$ because $\p$ is weakly homogeneous. We would
  like to argue that the stable core of $L[r][H]$ is $L$. By
  Corollary~\ref{cor:forcingabsolutness}, the stable core of $L[r]$ is
  $L$. So it remains to argue that forcing with $\p$ does not change
  the stable core. The forcing $\p$ preserves cardinals, so the strong
  limit cardinals of $L[r]$ are the same as in $L[r][H]$. By
  Corollary~\ref{cor:forcingabsolutness}, only triples
  $(n,\aleph_\omega,\gamma)$ with $n\geq 2$ in $S$ can be affected by
  $\p$. But for $n\geq 2$, $(n,\aleph_\omega,\gamma)$ can never make
  it into any stability predicate because $H_{\aleph_\omega}$ believes
  that there are no limit cardinals and $H_\gamma$ sees
  $\aleph_\omega$.
\end{proof}

We end the section with a brief description of a class forcing notion $\p$ making no use of the stability predicate such that $V$ is a class generic extension of \hbox{$\la
\HOD,\in, \p\ra$} (this possibility was first suggested by Woodin). Conditions in $\p$ are triples
$(\alpha,\varphi,\gamma)$, where $\alpha<\gamma$ are ordinals,
$\varphi$ is a formula with ordinal parameters below $\gamma$ which
defines in $V_\gamma$ a non-empty subset $X(\alpha,\varphi,\gamma)$ of
$P(\alpha)$. The ordering is given by
$(\alpha^*,\varphi^*,\gamma^*)\leq (\alpha,\varphi,\gamma)$ whenever
$\alpha\leq\alpha^*$ and for all $y\in
X(\alpha^*,\varphi^*,\gamma^*)$, $y\cap\alpha\in
X(\alpha,\varphi,\gamma)$. It is not difficult to see that if $A$ is
an $\Ord$-Cohen generic class of ordinals, then the collection
$G(A)=\{(\alpha,\varphi,\gamma)\in\p\mid A\cap\alpha\in
X(\alpha,\varphi,\gamma)\}$ is $\p$-generic over $V$. But since we can
easily recover $A$ from $G(A)$ and clearly $V=L[A]$, we have that
$V=L[G(A)]$. Unlike the forcing in \cite{Friedman:StableCore}, $\p$ does not have the $\Ord$-cc.

\section{Coding into the stable core over $L$}\label{sec:coding}
We will argue that any set added generically over $L$ can be coded
into the stable core of a further forcing extension. It is easiest to
code into the strong limit cardinals (because these are always
definable in the stable core), but we will show that we can actually
code into any $m$-th slice $S_m$ of the stability predicate.

\begin{thm}\label{th:addGenericSetCoreOverL}
Suppose $\p\in L$ is a forcing notion and $G\subseteq \p$ is $L$-generic. Then for every $m\geq 1$, there is a further forcing extension $L[G][H]$ such that $G\in L[S^{L[G][H]}_m]$.
\end{thm}
\begin{proof}
We can assume via coding that $G\subseteq\kappa$ for some cardinal $\kappa$. Also, since $\p$ is a set forcing, $\GCH$ holds on a tail of the cardinals in $L[G]$, and so on a tail, the strong limit cardinals coincide with the limit cardinals. Also, on a tail, $S^L$ agrees with $S^{L[G]}$ by Corollary~\ref{cor:forcingabsolutness}.

We work in $L$.   High above $\kappa$, we will define a sequence
$\la (\beta_\xi,\beta_\xi^*)\mid \xi<\kappa\ra$ of \emph{coding pairs}
such that $(\beta_\xi,\beta_\xi^*)\in S^L_m$. The coding forcing
$\mathbb C$ will be defined so that if $H\subseteq\mathbb C$ is
$L[G]$-generic, then we will have $\xi\in G$ if and only if
$(\beta_\xi,\beta_\xi^*)\in S^{L[G][H]}_m$. Since $L[S^{L[G][H]}_m]$ can
construct $L$, it will have the sequence of the coding pairs as well
as $S^L_m$, so that all the information put together will allow it to
recover $G$.

Call a strong limit cardinal $\alpha$ $m$-\emph{stable} if
$H_\alpha\prec_{\Sigma_m} L$. Observe that there is a proper class of
$m$-stable cardinals and if $\alpha$ and $\beta$ are both $m$-stable,
then the pair $(\alpha,\beta)\in S_m^L$. Let $\delta_0$ be the least
strong limit cardinal above $\kappa$. Let $\beta_0$ be the least
$m$-stable cardinal above $\delta_0$ of cofinality $\delta_0^+$ and
let $\beta_0^*$ be the least $m$-stable cardinal above $\beta_0$. Now
supposing we have defined the pairs $(\beta_\eta,\beta_\eta^*)$ of
$m$-stable cardinals for all $\eta<\xi$, let $\delta_\xi$ be the
supremum of the $\beta_\eta^*$ for $\eta<\xi$, let $\beta_\xi$ be the
least $m$-stable cardinal above $\delta_\xi$ of cofinality
$\delta_\xi^+$ and let $\beta_\xi^*$ be the least $m$-stable cardinal
above $\beta_\xi$. In particular, $\beta_\xi >\delta_\xi^+$ since, by
$m$-stability, $\beta_\xi$ is a strong limit cardinal. Note that the
sequence $\la (\beta_\xi,\beta_\xi^*)\mid \xi<\kappa\ra$ is
$\Sigma_{m+1}$-definable over $L$. Note also that
$\beta_\eta<\beta_\eta^*<\beta_\xi<\beta_\xi^*$ for all
$\eta<\xi<\kappa$ and for limit $\lambda<\kappa$,
$\beta_\lambda>\bigcup_{\xi<\lambda}\beta_\xi$, so that the sequence
of the $\beta_\xi$ will be purposefully discontinuous. Since the
forcing $\p$ is small relative to $\delta_0$, by
Corollary~\ref{cor:forcingabsolutness}, the coding pairs
$(\beta_\xi,\beta_\xi^*)\in S_m^{L[G]}$.

Now for $\xi<\kappa$, let $\mathbb C_\xi$ be the following forcing. If
$\xi\in G$, then $\mathbb C_\xi$ is the trivial forcing. If
$\xi\notin G$, then $\mathbb C_\xi=\Coll(\delta^+_\xi,\beta_\xi)$. Let
$\mathbb C$ be the full support product
$\Pi_{\xi<\kappa}\mathbb C_\xi$ and let $H\subseteq\mathbb C$ be
$L[G]$-generic.

Let's check that $\mathbb C$ collapses the minimum number of
cardinals, namely $\mathbb C$ collapses a cardinal $\delta$ if and
only if there is a non-trivial forcing stage $\xi$ such that
$\delta_\xi^+<\delta\leq \beta_\xi$. For every $\xi<\kappa$, the
forcing $\mathbb C$ factors as
$\Pi_{\eta<\xi}\mathbb C_\eta\times\Pi_{\xi\leq\eta<\kappa}\mathbb
C_\eta$, where the second part is $\lt\delta_\xi^+$-closed (using
full support), and so cannot collapse any cardinals
$\leq\delta^+_\xi$. Observe next that the forcing
$\Coll(\delta^+_\xi,\beta_\xi)$ has size
$\beta_\xi^{\delta_\xi}=\beta_\xi$ because
$\text{cf}(\beta_\xi)>\delta_\xi$ by our choice of $\beta_\xi$, and so cannot
collapse any cardinal $\geq\beta_\xi^+$. It follows that the forcing
$\mathbb C$ cannot collapse any
$\delta\in (\beta_\xi, \delta_{\xi+1}^+]$. It remains to show that
$\delta_\lambda$ and $\delta_\lambda^+$ for a limit $\lambda$ are
preserved. By what we already showed, $\delta_\lambda$ is a limit of
cardinals in the forcing extension, and therefore remains a
cardinal. Also, by what we already showed, if $\delta^+_\lambda$ is
collapsed, then it must be collapsed to $\delta_\lambda$. Suppose this
happens and fix a bijection $f:\delta_\lambda\to \delta_\lambda^+$ in
the forcing extension. We can let $f=\bigcup_{\xi<\lambda}f_\xi$,
where $f_\xi:\gamma_\xi\to \delta_\lambda^+$ and the $\gamma_\xi$ are
cofinal in $\delta_\lambda$. Each function $f_\xi$ must be added by
some initial segment of $\Pi_{\xi<\lambda}\mathbb C_\xi$ by closure,
and therefore its range must be bounded in $\delta_\lambda^+$. Now
build a descending sequence of conditions
$\la p_\xi\mid\xi<\lambda\ra$ in $\Pi_{\xi<\lambda}\mathbb C_\xi$ such
that $p_\xi$ decides the bound on the range of $f_\xi$. But then any
condition $p$ below the entire sequence forces that $f$ is bounded in
$\delta_\lambda^+$, which is the desired contradiction.

By the following claim, the forcing $\mathbb C$
also preserves $\GCH$ where the coding forcing takes place, so the
strong limit cardinals of $L[G][H]$ are precisely the limit cardinals
there.

\begin{claim}\label{cl:codingPreservesGCH}
  The $\GCH$ continues to hold on the part where it holds in $L[G]$ in
  the forcing extension $L[G][H]$ by $\mathbb C$.
\end{claim}
\begin{proof}
  By closure, it is clear that whereever the $\GCH$ held below
  $\delta_0^+$, it will continue to hold. Since $G \subseteq \kappa$,
  $\GCH$ holds in $L[G]$ above $\delta_0$.

  If there is trivial forcing at stage 0, then the $\GCH$ holds at
  $\delta_0^+$ in $L[G][H]$. So suppose that
  $\mathbb C_0=\Coll(\delta_0^+,\beta_0)$ is a non-trivial
  stage. Recall that $\Coll(\delta_0^+,\beta_0)$ has size $\beta_0$ so
  that there are $\beta_0^+$-many nice names for subsets of
  $\delta_0^+$ (and of course in $L[G][H]$,
  $(\delta_0^+)^+=(\beta_0^+)^{L[G]}$), which shows that the $\GCH$ holds
  at $\delta_0^+$ in $L[G][H]$ in this case as well.

  Now suppose inductively that the $\GCH$ holds up to some cardinal
  $\rho$. If $\rho=\delta_\xi^+$ for a successor ordinal $\xi$, we
  repeat the argument for $\xi=0$. If
  $\delta_\xi^+<\rho<\delta_{\xi+1}^+$ and there was non-trivial
  forcing at stage $\xi$, then $\beta_\xi<\rho<\delta_{\xi+1}^+$, and
  so the $\GCH$ continues to hold because the initial forcing is small
  relative to $\rho$ and the tail forcing is closed. Next, suppose
  $\rho=\delta_\lambda^+$ for a limit cardinal $\lambda<\kappa$. Since
  $\lambda$ is a limit, the initial segment forcing
  $\Pi_{\xi<\lambda}\mathbb C_\xi$ has size at most
  $\delta_\lambda^+$. This means that there are
  $(\delta^+_\lambda)^+$-many nice-names for subsets of
  $\delta_\lambda^+$, so that the $\GCH$ holds at
  $\rho=\delta_\lambda^+$. Finally, suppose
  $\rho=\delta_\lambda$. Each $A\subseteq\delta_\lambda$ is uniquely
  determined by the sequence $\la A_\xi\mid\xi<\lambda\ra$ with
  $A_\xi=A\cap\beta_\xi$. Let $\dot f_\xi$ be a name for an injection
  from $P(\beta_\xi)$ into $\delta_\lambda$, which exists since, by
  assumption, the $\GCH$ holds below $\delta_\lambda$ in $L[G][H]$. Let's
  argue that every sequence $\la A_\xi\mid\xi<\lambda\ra$ such that
  $A_\xi\subseteq \beta_\xi$ in the extension has a name of the form
  $\dot A$, where $\dot A(\xi)=\dot f^{-1}_\xi(\gamma)$ for some
  $\gamma\in\delta_\lambda$. Let $\dot B$ be any name for the sequence
  $\la A_\xi\mid\xi<\lambda\ra$ and $p'\in H$ be a condition forcing
  that $\dot B$ is a sequence of the right form. Below $p'$, we build
  a descending sequence $p_\xi$ for $\xi<\lambda$ of conditions
  deciding that $\dot B(\xi)=\dot f_\xi^{-1}(\gamma_\xi)$ for some
  fixed $\gamma_\xi<\delta_\lambda$. By closure, there is some $p$
  below the entire sequence. So by density, there is some such
  $p\in H$. It follows that there are at most as many subsets of
  $\delta_\lambda$ in the extension as there are functions
  $f:\lambda\to\delta_\lambda$ in the ground model, and there are
  $\delta_\lambda^+$-many such functions.
\end{proof}

Now we will argue that the pair
$(\beta_\xi,\beta_\xi^*)$ belongs in $S^{L[G][H]}_m$ if and only if $\xi\in G$.
If $\xi\notin G$, then $\beta_\xi$ is not even a cardinal in
$L[G][H]$, and therefore certainly
$(\beta_\xi,\beta_\xi^*)\notin S^{L[G][H]}_m$. Suppose that
$\xi\in G$, so that there is trivial forcing at stage $\xi$. By what
we already argued about which cardinals are collapsed in $L[G][H]$, it
follows that $\beta_\xi$ and $\beta_\xi^*$ are limit cardinals
there. Let $\mathbb C_\sm=\Pi_{\eta<\xi}\mathbb C_\eta$ and
$\mathbb C_\tail=\Pi_{\xi<\eta<\kappa}\mathbb C_\eta$, and note that
since there is no forcing at stage $\xi$, $\mathbb C$ factors as
$\mathbb C_\sm\times\mathbb C_\tail$. Let $H_\sm\times H_\tail$ be the
corresponding factoring of the generic filter $H$. Since
$\mathbb C_\tail$ is $\leq\beta^*_\xi$-closed, we have that
$H_{\beta_\xi}^{L[G][H]}=H_{\beta_\xi}^{L[G][H_\sm]}$ and
$H_{\beta_\xi^*}^{L[G][H]}=H_{\beta_\xi^*}^{L[G][H_\sm]}$. By
Proposition~\ref{prop:liftingElementarity}~(3),
$H_{\beta_\xi}^{L[G][H_\sm]}$ satisfies $\Sigma_m$-collection and by
Proposition~\ref{prop:liftingElementarity}~(2),
$H_{\beta_\xi}^{L[G][H_\sm]}\prec_{\Sigma_m}
H_{\beta_\xi^*}^{L[G][H_\sm]}$.
\end{proof}
It follows from Theorem~\ref{th:addGenericSetCoreOverL} that
(consistently) the stable core is not a fine-structural or in any
sense canonical inner model. Among the numerous corollaries of
Theorem~\ref{th:addGenericSetCoreOverL} are the following.
\begin{cor}$\,$
\begin{itemize}
\item[(1)] The $\GCH$ can fail on an arbitrarily large initial segment
  of the regular cardinals in the stable core.
\item[(2)] An arbitrarily large ordinal of $L$ can be countable in the stable core.
\item[(3)] ${\rm MA}+\neg{\rm CH}$ can hold in the stable core.
\end{itemize}
\end{cor}
\begin{proof}
  For (1), we force over $L$ to violate the $\GCH$ on an initial
  segment of the regular cardinals, and then code all the subsets we
  add into the stable core of the forcing extension by the coding
  forcing $\mathbb C$. For (2), we force over $L$ to collapse the
  ordinal, and then code the collapsing map into the stable core of
  the forcing extension by the coding forcing $\mathbb C$. For (3), we
  force Martin's Axiom with $2^\omega=\kappa$ to hold over $L$, and let $L[G]$ be the forcing extension. We
  then code the $G$ into the stable core of the forcing extension high above $\kappa$. Any partial order $\p$ on $\kappa$ in the stable core of the coding extension already exists in $L[G]$ and therefore $G$ will have added a partial generic filter for it.
\end{proof}

\begin{thm}\label{th:GCHfails}
It is consistent that the $\GCH$ fails at all regular cardinals in the stable core.
\end{thm}
\begin{proof}
The idea will be to force the $\GCH$ to fail at all regular cardinals over $L$, and then use $\Ord$-many coding pairs to code all the added subsets into the stable core of a forcing extension. In this argument, we will code into the limit cardinals by using generalized Cohen forcing instead of the collapse forcing.

In $L$, let $\p$ be the Easton support $\Ord$-length product forcing with $\Add(\kappa,\kappa^{++})$ at every regular cardinal $\kappa$, and let $G\subseteq \p$ be $L$-generic. Standard arguments show that in $L[G]$, $2^\kappa=\kappa^{++}$ for every regular cardinal $\kappa$, while the $\GCH$ continues to hold at singular cardinals (see, for example, \cite{Jech:SetTheory}). Since $\la L[G],\in, G\ra$ has a definable global well-order, we can assume via coding that $G\subseteq \Ord$ (and define the coding forcing in this expanded structure).

We first work in $L$. Let $\delta_0$ be the least strong limit
cardinal. Above $\delta_0$, we will define a sequence
$\la (\beta_\xi,\beta_\xi^*)\mid\xi\in\Ord\ra$ of \emph{coding pairs}
of strong limit cardinals. Let $\beta_0<\beta_0^*$ be the next two strong
limit cardinals above $\delta_0$. Now supposing we have defined the
pairs $(\beta_\eta,\beta_\eta^*)$ of strong limit cardinals for all
$\eta<\xi$, let $\delta_\xi$ be the supremum of the $\beta_\eta^*$ for
$\eta<\xi$ and let $\beta_{\xi}<\beta_{\xi}^*$ be the next two
strong limit cardinals above $\delta_\xi$. Observe that every strong
limit cardinal of $L$ remains a strong limit in $L[G]$, and so in
particular, the elements $\beta_\xi$ and $\beta_\xi^*$ of the coding
pairs are strong limits in $L[G]$.

For each ordinal $\xi$, let $\mathbb C_\xi$ be the following
forcing. If $\xi\in G$, then $\mathbb C_\xi$ is the trivial
forcing. So suppose that $\xi\notin G$. If $\delta_\xi$ is singular,
we let $\mathbb C_\xi=\Add(\delta^+_\xi,\beta_\xi)$ (the partial order
to add $\beta_\xi$-many Cohen subsets to $\delta^+_\xi$ with bounded
conditions), and otherwise, we let
$\mathbb C_\xi=\Add(\delta_\xi,\beta_\xi)$. Let's argue that all
forcing notions $\mathbb C_\xi$ are cardinal preserving. If
$\delta_\xi$ is singular, then $\GCH$ holds at $\delta_\xi$, and
therefore $\Add(\delta^+_\xi,\beta_\xi)$ has the
$(2^{\lt\delta_\xi^+})^+=(2^{\delta_\xi})^+=\delta_\xi^{++}$ chain
condition, which means that it preserves all cardinals. If
$\delta_\xi$ is regular, then it is inaccessible because it is always
a limit cardinal, and therefore $\Add(\delta_\xi,\beta_\xi)$ preserves
all cardinals. Obviously, every non-trivial forcing $\mathbb C_\xi$
destroys the strong limit property of $\beta_\xi$ in the forcing
extension.

Let $\mathbb C$ be the $\Ord$-length Easton support product $\Pi_{\xi\in\Ord}\mathbb C_\xi$. Let's argue that the forcing notion $\mathbb C$ is also cardinal preserving. Observe first that if $\delta_\xi$ is singular, then the initial segment  $\Pi_{\eta<\xi}\mathbb C_\eta$ has size $\delta_\xi^{\delta_\xi}=\delta_\xi^+$ since $\GCH$ holds at $\delta_\xi$. If $\delta_\xi$ is regular, then $\delta_\xi$ is inaccessible, so that conditions in $\Pi_{\eta<\xi}\mathbb C_\eta$ are bounded, and hence $\Pi_{\eta<\xi}\mathbb C_\eta$ has size $\delta_\xi^{\lt\delta_\xi}=\delta_\xi$.  Now we can argue that if $\delta_\xi^+<\gamma<\delta_{\xi+1}$ is a cardinal, then it remains a cardinal in the forcing extension by $\mathbb C$ because by previous calculations, the initial segment $\Pi_{\eta<\xi}\mathbb C_\eta\times \mathbb C_\xi$  cannot collapse $\gamma$, and the tail forcing is highly closed. Cardinals of the form $\delta_{\xi+1}$ cannot be collapsed because the successor stage forcings are cardinal preserving. It remains to consider cardinals of the form $\delta_\lambda$ and $\delta_\lambda^+$ for a limit cardinal $\lambda$. By what we already showed, $\delta_\lambda$ is a limit of cardinals in the forcing extension, and hence must be a cardinal itself. If $\delta_\lambda$ is regular, then it is inaccessible, and hence the initial segment $\Pi_{\xi<\lambda}\mathbb C_\xi$ is too small to collapse $\delta_\lambda^+$. So suppose that $\delta_\lambda$ is singular with $\text{cof}(\delta_\lambda)=\mu<\delta_\lambda$. By regrouping the product, we can view the forcing $\Pi_{\xi<\lambda}\mathbb C_\xi$ as a product of length $\mu$, which is $\lt\mu$-closed on a tail. Thus, an analogous argument to the one given in the proof of Theorem~\ref{th:addGenericSetCoreOverL} shows that $\delta_\lambda^+$ cannot be collapsed to $\delta_\lambda$ in this case, completing the proof that $\mathbb C$ is cardinal preserving. In particular, this implies that $\GCH$ continues to fail at all regular cardinals in any forcing extension by $\mathbb C$.

Let $H\subseteq\mathbb C$ be $L[G]$-generic. For each $\xi\in\Ord$, we can factor $\mathbb C$ as the product $\Pi_{\eta<\xi}C_\eta\times\Pi_{\xi\leq\eta}C_\eta$, where the tail forcing $\Pi_{\xi\leq \eta}C_\eta$ is $\lt\delta_\xi$-closed since we used Easton support. Note that since $\mathbb C$ is a progressively closed class product, it preserves $\ZFC$ to the forcing extension $L[G][H]$.

Suppose $\xi<\kappa$ is a trivial stage of forcing in $\mathbb C$. Let $\mathbb C_\sm=\Pi_{\eta<\xi}\mathbb C_\eta$ and $\mathbb C_\tail=\Pi_{\xi<\eta}\mathbb C_\eta$. The forcing $\mathbb C_\sm$ has size at most $\delta_\xi^+$, and therefore cannot destroy the strong limit property of $\beta_\xi$ and $\beta_\xi^*$, and neither can $\mathbb C_\tail$, which is $\lt\beta_\xi^*$-closed. It follows that $\beta_\xi$ and $\beta_\xi^*$ remain strong limits in $L[G][H]$.
\end{proof}
Let's see what it would take to code subsets added by $G$ into the
$m$-th slice $S_m$ of the stability predicate. The main problem is that if $\alpha$ is a singular cardinal, then $\p\restrict\alpha$ has unbounded support in $\alpha$, and therefore $\p\restrict\alpha$ is not a class forcing over $H_\alpha$, which prevents us from using standard lifting arguments to go from $H_\alpha^L\prec_{\Sigma_m}H_\beta^L$ to $H_\alpha^{L[G]}\prec_{\Sigma_m} H_\beta^{L[G]}$. The construction would go through for $m$, if we assume that $L$ has a proper class of inaccessible cardinals $\alpha$ such that $H_\alpha^L\prec_{\Sigma_{m+1}} L$. The class forcing $\p$ is $\Delta_2$-definable, so the forcing relation for $\Sigma_m$-formulas is $\Sigma_{m+1}$-definable. Using this, we can argue that if $H_\alpha^L\prec_{\Sigma_{m+1}} H_\beta^L$, then $H_\alpha^{L[G]}=H_\alpha^L[G]\prec_{\Sigma_{m}}H_\beta^L[G]=H_\beta^{L[G]}$.

Finally, let's note that if we only wanted the $\GCH$ to fail cofinally, then we could force in a single step to add $\kappa^{++}$-many subsets to some $\kappa$, followed by the forcing to code the sets into the stable core, and do this for cofinally many cardinals, spacing them out enough to prevent interference.
\section{Measurable cardinals in the stable core}\label{sec:measurables}
In \cite{KennedyMagidorVaananen:extendedLogics}, Kennedy, Magidor, and
V\"a\"an\"anen, studied properties of the model $\la
L[\Card],\in,\Card\ra$ for the class $\Card$ of cardinals of $V$. They
showed that if there is a measurable cardinal, then $L[\mu]$, the canonical model for a single measurable cardinal, is contained in $L[\Card]$. In particular, $L[\Card]^{L[\mu]}=L[\mu]$, which shows that $L[\Card]$ can have a measurable cardinal. Recently, Philip Welch showed that if $m_1^{\#}$ exists, then $L[\Card]$ is a certain Prikry-type forcing extension of an iterate of $m_1^{\#}$ adding Prikry sequences to all measurable cardinals in it \cite{Welch:ClubClassHartigQuantifierModel}. It follows from this that, in the presence of sufficiently large large cardinals, the model $L[\Card]$ satisfies the $\GCH$ and has no measurable cardinals, although it does have an inner model with a proper class of measurables.

We adapt techniques of \cite{KennedyMagidorVaananen:extendedLogics} to show that if there is a measurable cardinal, then, for every $m\geq 1$, $L[\mu]$ is contained in $L[S_m]$. In particular, $L[S^{L[\mu]}]=L[\mu]$, showing that the stable core can have a measurable cardinal. Indeed, we improve this result to show that the stable core can have a discrete proper class of measurable cardinals.

Let's start with the following easy proposition showing that if $0^\#$ exists, then it is in the stable core.
\begin{prop}\label{prop:zeroSharp}
If $0^\#$ exists, then $0^\#\in L[S_m]$ for every $m\geq 1$.
\end{prop}
\begin{proof}
Every $L[S_m]$ has many increasing $\omega$-sequences of $V$-cardinals, so fix some such sequence $\la \alpha_n\mid n<\omega\ra$. We have that $\varphi(x_0,\ldots,x_{n-1})\in 0^\#$ if and only if $L_{\alpha_n}\models\varphi(\alpha_0,\ldots,\alpha_{n-1})$.
\end{proof}

\begin{thm}\label{thm:LmuSC}
Suppose that $\kappa$ is a measurable cardinal and $L[\mu]$ is the canonical
  inner model with a normal measure $\mu$ on $\kappa$. Then
  $L[\mu]\subseteq L[S_m]$ for every $m\geq 1$.
\end{thm}
The proof of this theorem uses techniques from the proof of Kunen's Uniqueness Theorem (\cite{Kunen:dissertation} and \cite{Kunen:IteratedUltrapowers}, for a modern account,
see for example \cite[Theorem 19.14]{Jech:SetTheory}) and is following the idea of Theorem 9.1 in \cite{KennedyMagidorVaananen:extendedLogics}.
\begin{proof}[Proof of Theorem \ref{thm:LmuSC}]
  We will first argue that for some sufficiently large $\lambda$, the
  normal measure $\mu_\lambda$ on the $\lambda$-th iterated ultrapower
  of $L[\mu]$ by $\mu$ is in $L[S_m]$, and then find in $L[S_m]$ an
  elementary substructure of size $(\kappa^+)^V$ of an initial segment
  $L_\theta[\mu_\lambda]$ of the iterate that will collapse to $L_{\bar\theta}[\mu]$.

We can assume that $\mu\in L[\mu]$. We work in $V$ and fix $m\geq
1$. Let $\lambda \gg \kappa^{+}$ be a strong limit cardinal with
unboundedly many $\alpha$ such that $(\alpha,\lambda)\in S_m$. Let
$j_\lambda: L[\mu]\to L[\mu_\lambda]$ be the embedding given by the $\lambda$-th iterated ultrapower of $L[\mu]$ by $\mu$, so that in $L[\mu_\lambda]$, $\mu_\lambda$ is a normal measure on the cardinal $\lambda=j_\lambda(\kappa)$ (by \cite[Corollary 19.7]{Kanamori:HigherInfinite}, for all cardinals $\lambda>\kappa^+$, the $\lambda$-th element of the critical sequence is $\lambda$). Let $\la \kappa_\xi \st \xi <\lambda\ra$ be the critical sequence of the iteration by $\mu$. Finally, let $\cF$ denote the filter generated by the tails $$A_\xi=\{\eta < \lambda \st \xi \leq \eta \text{ such that } (\eta,\lambda)\in S_m\}$$ for $\xi < \lambda$. We will argue that $L[\mu_\lambda]=L[\cF]$. It will follow that $L[\mu_\lambda]\subseteq L[S_m]$ since $L[S_m]$ can compute $L[\cF]$ from $S_m$.

First, let's argue that $\mu_\lambda\subseteq \cF$. Suppose
$X \in \mu_\lambda$. Then there must be a $\zeta < \lambda$ such that
$\{ \kappa_\xi \st \zeta \leq \xi < \lambda\} \subseteq X$ (see
\cite[Lemma 19.5]{Kanamori:HigherInfinite}). As $\kappa_\eta = \eta$ for every
sufficiently large cardinal $\eta < \lambda$ (\cite[Corollary
19.7]{Kanamori:HigherInfinite}), it follows that
\[\{\eta <\lambda\st \zeta^\prime \leq \eta \text{ and } \eta \text{
    is a cardinal}\} \subseteq X\] for some $\zeta^\prime <
\lambda$. In particular, $A_{\zeta^\prime} \subseteq X$, and thus
$X \in \cF$. But now since $\mu_\lambda$ is an ultrafilter in
$L[\mu_\lambda]$ and $\cF$ is a filter, it follows that
$\cF\cap L[\mu_\lambda]\subseteq\mu_\lambda$ and hence
$\cF\cap L[\mu_\lambda]=\mu_\lambda$. From here it is not difficult to
see that $L[\mu_\lambda]=L[\cF]$, and hence
$L[\mu_\lambda]\subseteq L[S_m]$.

Now we will define in $L[S_m]$, a
sequence of length $(\kappa^+)^V$ whose elements will generate the
desired elementary substructure. Recall that if $\eta$ is a strong limit cardinal of cofinality greater
than $\kappa$ and moreover $\eta>\lambda$ (the length of the
iteration), then $j_\lambda(\eta)=\eta$ (see \cite[Corollary
19.7]{Kanamori:HigherInfinite}).  Let $\lambda^* \gg \lambda$ be a
strong limit cardinal of cofinality greater than $(\kappa^+)^V$ such that the set
$S_m^{\lambda^*} = \{ \eta \st (\eta,\lambda^*) \in S_m \}$ is unbounded in $\lambda^*$. Let
$\eta_0$ be the $(\kappa^+)^V$-th element of $S_m^{\lambda^*}$ above
$\lambda$. Inductively, let $\eta_{\xi+1}$ be the $(\kappa^+)^V$-th
element of $S_m^{\lambda^*}$ above $\eta_\xi$ and
$\eta_\delta = \bigcup_{\xi < \delta} \eta_\xi$ for limit ordinals
$\delta$. Let $A = \{ \eta_{\xi +1} \st \xi < (\kappa^+)^V\}$. As
$(\kappa^+)^V$ is regular (in $V$), it follows that
$\cf^V(\eta_{\xi+1})=(\kappa^+)^V$ for all $\eta_{\xi+1} \in
A$. Therefore each element of $A$ is fixed by the iteration embedding
$j_\lambda$.

Fix $\theta$ above the supremum of $A$. Let
$X\prec L_\theta[\mu_\lambda]$ be the Skolem hull of $\kappa\cup A$ in
$L_\theta[\mu_\lambda]$, and note that $X \in L[S_m]$. Let $N$ denote
the Mostowski collapse of $X$, and let
$$\sigma: N \rightarrow X \prec L_\theta[\mu_\lambda]$$ be the inverse
of the collapse embedding. Note that $\lambda$ is in $X$ as it is
definable as the unique measurable cardinal in
$L_\theta[\mu_\lambda]$. In fact, $\sigma(\kappa) = \lambda$ by the
following argument. As $X$ is generated by elements from
$j_\lambda\image L[\mu]$, it is contained in
$j_\lambda\image L[\mu]\prec L[\mu_\lambda]$. But there is no
$\gamma \in j_\lambda\image L[\mu]$ with $\kappa < \gamma < \lambda$, so
$\lambda$ has to collapse to $\kappa$. Finally, since $|A| = (\kappa^+)^V$
and $\sigma(\kappa)=\lambda$, the collapse $N$ has the form
$L_{\bar\theta}[\nu]$ with $\nu$ a normal measure on $\kappa$ and
$\bar\theta$ an ordinal of size $(\kappa^+)^V$. By Kunen's Uniqueness
Theorem (see for example \cite[Theorem 19.14]{Jech:SetTheory}),
$N = L_{\bar\theta}[\mu]$, and thus $L_{\bar\theta}[\mu] \in L[S_m]$. So
$L[\mu]\subseteq L[S_m]$, as desired.
\end{proof}

\begin{cor}\label{cor:SCmeascard} $\,$
  \begin{itemize}
  \item[(1)] We have $L[S^{L[\mu]}] = L[\mu]$. In particular, it is consistent
    that the stable core has a measurable cardinal.
  \item[(2)] Let $K^{DJ}$ denote the Dodd-Jensen core model below a
    measurable cardinal. Then $K^{DJ} \subseteq L[S]$, and hence
    $L[S^{K^{DJ}}]=K^{DJ}$.
  \item[(3)] If $0^\dagger$ exists, then $0^\dagger \in L[S]$.
  \end{itemize}
\end{cor}
\begin{proof} (1) follows immediately from Theorem \ref{thm:LmuSC} by applying it inside $V = L[\mu]$.

For (2) we first recall the definition of the Dodd-Jensen core model
$K^{DJ}$ from \cite{DoddJensen:CoreModel}. We call a transitive model
$M$ of the form $M=\la J_\alpha[U],\in,U\ra$ a \emph{Dodd-Jensen
  mouse} if $M$ satisfies that $U$ is a normal measure on some $\kappa
< \alpha$, all of the iterated ultrapowers of $M$ by $U$ are
well-founded, and $M$ has a fine structural property implying that $M = \Hull_1^M(\rho \cup p)$ (the $\Sigma_1$-Skolem closure of $M$) for some ordinal $\rho < \kappa$ and some finite set of parameters $p \subseteq \alpha$ (see \cite[Definition 5.4]{DoddJensen:CoreModel}). The Dodd-Jensen core model $K^{DJ} = L[\mathcal{M}]$, where $\mathcal{M}$ is the class of all such Dodd-Jensen mice (see \cite[Definition 6.3]{DoddJensen:CoreModel} or, for a modern account, \cite{Mitchell:BeginningInnerModelTheory}). So we need to argue that every such mouse $M$ is in $L[S]$. We essentially follow the proof of Theorem \ref{thm:LmuSC} to show that some $\lambda$-th iterate $M_\lambda$ of $M$ is in $L[S]$. Then we argue that
$M \in L[M_\lambda]$ as, by $\Sigma_1$-elementarity of $j_\lambda$, $M$ is isomorphic to $\Hull_1^{M_\lambda}(\rho \cup
  j_\lambda(p))$. Hence, $M \in L[S]$.

For (3), since the strong limit cardinals of $V$ are definable in $L[S]$, the result for $0^\dagger$ follows from Theorem \ref{thm:LmuSC} as in the proof of Proposition~\ref{prop:zeroSharp}.
\end{proof}

The argument above relied on the fact that for a certain club $C$ of
cardinals, $L[\mu]\subseteq L[C]$ and $L[C]\subseteq L[S]$. The
following result shows that this argument with one club $C$ cannot be
pushed further to show that stronger large cardinals are in the stable
core. Given a club $C$, we will denote by $\hat C$, the collection of all successor elements of $C$ together with its least element.

\begin{thm}[\cite{Welch:ClubClassHartigQuantifierModel}]\label{thm:L[C]}
Suppose that $C$ is a class club of uncountable cardinals. Then there is an $\Ord$-length iteration of the mouse $m_1^{\#}$ such that in the direct limit model $M_C$ (truncated at $\Ord$), the measurable cardinals are precisely the elements of $\hat C$.
\end{thm}
\begin{proof}
Let $\bar\kappa$ be the critical point of the top measure of $m_1^{\#}$. Let $\la \alpha_\xi\mid\xi\in\Ord\ra$ be the increasing enumeration
of $C$.

Iterate the first measurable cardinal $\kappa_0$ of $m_1^{\#}$
  $\alpha_0$-many times, so that $\kappa_0$ iterates to $\alpha_0$, and
  let $M_{\alpha_0}$ be the iterate. Since $m_1^{\#}$ is countable, $M_{\alpha_0}$ has
  cardinality $\alpha_0$, and hence the critical point of the top measure $\bar\kappa_{\alpha_0}$, the image of $\bar\kappa$ in $M_{\alpha_0}$,
  is below $\alpha_1$. In particular, the next measurable cardinal
  $\kappa_1$ above $\alpha_0$ in $M_{\alpha_0}$ is below $\alpha_1$ and we can
  iterate it to $\alpha_1$ by iterating it $\alpha_1$-many
  times. Repeat this for all successor ordinals $\xi$ and take direct
  limits along the iteration embeddings at limit stages with the
  following exception.

  Suppose we have carried out the construction for a limit $\xi$-many steps
  resulting in the model $M_{\xi}$, where $\bar\kappa_{\xi}$ is the critical point of the top measure, in which the measurable cardinals limit up to $\bar\kappa_\xi$. In this case, we must have $\xi=\alpha_\xi$. When this happens we have run out of room and don't have a measurable cardinal to iterate to the next element $\alpha_{\xi+1}$ of $\hat C$. To make more space, in the next step, we iterate up the top measure to obtain the model $M_{\xi+1}$ with more measurable cardinals. By cardinality considerations, the critical point $\bar\kappa_{\xi+1}$ of the top measure is obviously below $\alpha_{\xi+1}$. Hence, we can continue the construction, iterating the least measurable cardinal $\kappa_{\xi+1}$ above
  $\xi$ in $M_{\xi+1}$ to $\alpha_{\xi+1}$.

  Let $M$ be the resulting model obtained as the direct limit along
  the iteration embeddings and let $M_C$ be $M$ truncated at $\Ord$, which is the cardinal on which the top measure of $M$ lives. The construction ensures that we hit every element of $\hat C$ along the way, so that the measurable
  cardinals in $M_C$ are exactly the elements $\hat C$.
\end{proof}
A more elaborate version of this iteration argument is
going to be used in Section \ref{sec:Lcard} to generalize the results of \cite{Welch:ClubClassHartigQuantifierModel} to a finite number of specially nested clubs.
\begin{cor}\label{cor:mouseNotInL[C]}
  If $C$ is a class club of uncountable $V$-cardinals, then $m_1^{\#}\notin L[C]$.
\end{cor}

\begin{proof}
Suppose towards a contradiction that $m_1^{\#}\in L[C]$.  Iterate $m_1^{\#}$ inside $L[C]$ to a model $M$ as in the proof of Theorem~\ref{thm:L[C]}, and let $M_C$ be the truncation of $M$ at $\Ord$. In
  particular, $C$ is definable in $M_C$ by considering the closure of
  its measurable cardinals. It follows that $L[C]$ is a definable
  sub-class of $M_C$, so that $L[C] = M_C$. But this is impossible
  because $m_1^{\#}$ is a countable model (in $L[C]$), which means, in
  particular, that $\omega_1^{m_1^{\#}}=\omega_1^{M_C}$ is countable in
  $L[C]$.
\end{proof}

In the last section of the article we will show that, unlike $L[\Card]$, the stable core, given sufficiently large large cardinals, can have inner models with a proper class of $n$-measurable cardinals for any $n<\omega$.

Now we can say more precisely what Welch showed about the model
$L[\Card]$ in \cite{Welch:ClubClassHartigQuantifierModel}. Let $M_C$,
for the class club $C$ of limit cardinals, be the iterate of
$m_1^{\#}$ (truncated at $\Ord$) obtained as in the proof of Theorem
\ref{thm:L[C]} in which the measurable cardinals are precisely the
cardinals of $V$ of the form $\aleph_{\omega\cdot\alpha+\omega}$
(namely elements of $\hat C$). Let $U_\alpha\subseteq P^{M_C}(\aleph_{\omega\cdot\alpha+\omega})$ be the iteration measures on $\aleph_{\omega\cdot\alpha+\omega}$ in $M_C$, and note that a subset of $\aleph_{\omega\cdot\alpha+\omega}$ in $M_C$ is in $U_\alpha$ if and only if it contains some tail of the cardinals. Let $\vec U=\la U_\alpha\mid\alpha\in\Ord\ra$. The model $M_C$ has the form $L[\vec U]$ because it is an iterate of the mouse $m^{\#}_1$. Let $W_\alpha\subseteq P^{L[\Card]}(\aleph_{\omega\cdot\alpha+\omega})$ consist of
all subsets of $\aleph_{\omega\cdot\alpha+\omega}$ in $L[\Card]$ containing some tail of cardinals $\aleph_{\omega\cdot\alpha+n}$, $n<\omega$, and let $\vec W$ be the sequence of the $W_\alpha$. Now it is easy to see that $L[\vec W]=L[\vec U]$, and hence, since $\vec W$ is definable in $L[\Card]$, $M_C$ is contained in $L[\Card]$. Let $f$ be a function
on $\hat C$ such that
$$f(\omega\cdot\alpha+\omega)=\la \omega\cdot\alpha+n\mid n<\omega\ra.$$ We clearly have that $L[\Card]=L[f]$, and also
$L[f]=L[\vec W][f]=M_C[f]$ because the sequence $\vec W$ can be
recovered from $f$. Thus, $L[\Card]=M_C[f]$, and it turns out that in
some sense which we will explain in detail in
Section~\ref{sec:characterization}, $M_C[f]$ is a Prikry-type forcing
extension of $M_C$ adding Prikry sequences to all its measurable
cardinals.

Note that in this
construction we iterated the measurable cardinals to elements of $\hat C$, where $C$ is the club of limit cardinals, instead of to all successor cardinals,
because we need to have enough cardinals in between to be able to use
them to define the measures in $\vec U$, so that $L[\vec U]$ is
contained in $L[C]$. If on the other hand, we iterate the measurable
cardinals to all successor cardinals,
then we can get an inclusion in the other direction: $L[C]$ is
contained in the iterate $M_C$.

We will now generalize the result that the stable core of $L[\mu]$ is equal to $L[\mu]$ to show that if $\vec U$ is a discrete proper class of normal measures, then stable core of $L[\vec U]$ is $L[\vec U]$. It follows that the stable core can have a proper class of measurable cardinals.

\begin{thm}
If $\vec U$ is a discrete proper class sequence of normal measures, then $L[S]^{L[\vec U]} = L[\vec U]$. In particular, it is consistent that the stable core has a proper class of measurable cardinals.
\end{thm}

\begin{proof}
  Let $\vec U$ be a discrete proper class sequence of normal measures and work in $V = L[\vec U]$. Consider the stable core $L[S]$ and the corresponding core models $K_0 = K^{L[S]}$ and $K = K^V$. Recall that all measurable cardinals in $V$ are measurable in $K$ as witnessed by restrictions of the measures in $\vec U$ and therefore $V$ and $K$ have the same universes. Compare $K_0$ and $K$ in $V$. As both are proper class models they have a common iterate $K^*$.
  \setcounter{case}{0}
  \begin{case}\label{case:1}
    The $K$-side of the coiteration drops.
  \end{case}
  Then $K^*$ is the $\Ord$-length iterate of some mouse\footnote{Note that when we say \emph{mouse} here we mean a fine structural iterable premouse which has partial measures on its sequence as for example in \cite[Section 4]{Zeman:InnerModelsLargeCardinals}.} $\cM$ which appears after the last drop on the $K$-side of the coiteration such that $K^*$ is the result of hitting a measure on some $\kappa$ in $\cM$ and its images (truncated at $\Ord$). The successive images of $\kappa$ form a $V$-definable club $D_0$ of ordinals which are regular cardinals in $K^*$. The $K_0$-side of the coiteration does not drop, so there is an iteration map $\pi_0 \colon K_0 \rightarrow K^*$ and the ordinals $\alpha$ such that $\pi_0 \image \alpha \subseteq \alpha$ form a $V$-definable club $D_1$. Let $D = D_0 \cap D_1$. Note that the
iteration of $K_0$ has set-length, since the measures on the
$K$-side, and therefore also those on the $K_0$-side, are
bounded by the measurable $\kappa$ which is sent to $\Ord$ on the
$K$-side of the iteration (by the discreteness of the measure
sequence). It follows that for some $\delta$,
all elements of $D$ of cofinality at least $\delta$ are fixed by
the iteration map $\pi_0$.
 
  Let $n<\omega$ be large enough such that $D$ is $\Sigma_n$-definable in $V$. Recall from Proposition~\ref{prop:SigmaElementaryClubs} that the class club $C_n$ consisting of all strong limit $\beta$ such that $H_\beta \prec_{\Sigma_n} V$ is definable from $S$. Let $\beta \in C_n$ be sufficiently large. In $V$, $D$ is cofinal in $\Ord$. Therefore, in $H_\beta$, $D \cap H_\beta$ is cofinal in $\beta$, and hence $\beta \in D$. So a tail of $C_n$ is contained in $D$ and there is a $\delta$-sequence of adjacent elements of $C_n$ contained in $D$ such that its limit $\lambda$ is singular of uncountable cofinality in $L[S]$. But $\lambda \in D$ and all elements of $D$ are regular in $K^*$. As $\pi_0(\lambda) = \lambda$, $\lambda$ is also regular in $K_0$, contradicting the covering lemma for sequences of measures in $L[S]$ (see \cite{Mi84} and \cite{Mi87}).

  \begin{case}
    The $K_0$-side of the coiteration drops.
  \end{case}
  Let $\cN$ be the model on the $K_0$-side of the coiteration after the last drop. Then $\cN \cap \Ord < \Ord$, but the coiteration of $\cN$ and an iterate $K^\prime$ of $K$ results in the common proper class iterate $K^*$. The iteration from $K$ to $K^\prime$ is non-dropping and hence $K^\prime$ is universal. But this contradicts the fact that the coiteration of $\cN$ and $K^\prime$ does not terminate after set-many steps.

  \begin{case}
    Both sides of the coiteration do not drop, i.e., there are elementary embeddings $\pi_0 \colon K_0 \rightarrow K^*$ and $\pi \colon K \rightarrow K^*$.
  \end{case}
  As $K^*$, and hence $K_0$, has a proper class discrete sequence of measures it is universal in $V = L[\vec U]$. Therefore, in fact, $K_0 = K^*$ by the proof of Theorem 7.4.8 in \cite{Zeman:InnerModelsLargeCardinals}. Finally, we argue that $K$ cannot move in the iteration to $K_0$. Suppose this is not the case and let $U$ on $\kappa$ be the first measure that is used. Let $\kappa^*$ be the image of $\kappa$ in $K_0$. For some large enough $n < \omega$, $C_n$ can define a proper class $C^*$ of fixed points of $\pi$ as follows. There is a $V$-definable club $C$ of ordinals $\alpha$ such that $\pi \pwimg \alpha \subseteq \alpha$. As in Case \ref{case:1}, a tail of $C_n$ is contained in $C$. Let $\beta \in C_n$ be an arbitrary element of that tail and let $\gamma$ be the $\omega$-th element of $C_n$ above $\beta$. Then $\gamma$ is a closure point of $\pi$ and $\cf(\gamma) = \omega$ in $V$ and hence in $K$, since the universes of $V$ and $K$ agree. So the iteration map is continuous at $\gamma$ and therefore $\pi(\gamma) = \gamma$.

  Let $\bar{K}_0$ be the transitive collapse of $\Hull^{K_0}(\kappa \cup \{\kappa^*\} \cup C^*)$. Then $\bar{K}_0$ has a proper class of measurable cardinals including $\kappa$. In particular, $\bar{K}_0$ is a universal weasel, and hence an iterate of $K$, where the first measure used in the iteration has critical point above $\kappa$. Therefore $\bar{K}_0$ and hence $L[S]$ and $K_0$ contain the measure $U$ on $\kappa$, contradicting the fact that this measure was used in the iteration.

  Therefore, we obtain that $K_0 = K$. As $L[\vec U]$ can be reconstructed from $K$ it follows that $L[S] = L[\vec U]$.
\end{proof}

The arguments of Section~\ref{sec:coding} generalize directly to
coding sets added generically over $L[\mu]$ into the stable core of a
further forcing extension. If the forcing adding the generic sets is
either small relative to the measurable cardinal $\kappa$ of $L[\mu]$
or $\leq\kappa$-closed, and the coding is done high above $\kappa$,
then the stable core of the coding extension will continue to think
that $\kappa$ is measurable.
\begin{thm}\label{thm:codingintoLmu}
  Suppose $V=L[\mu]$. If $\p$ is a forcing notion of size less than
  $\kappa$ or $\p$ is $\leq\kappa$-closed and $G\subseteq \p$ is
  $V$-generic, then there is a further forcing extension $V[G][H]$
  such that $G\in L[S^{V[G][H]}]$ and $\kappa$ remains a measurable
  cardinal there.
\end{thm}
\begin{proof}
  Suppose $\p$ is a small forcing. By the \Levy-Solovay theorem,
  $\kappa$ remains measurable in $V[G]$, as witnessed by the normal
  measure $\nu$ on $\kappa$ such that $A\in \nu$ if and only if there
  is $\bar A\in \mu$ with $\bar A\subseteq A$. Since the coding
  forcing defined in the proof of
  Theorem~\ref{th:addGenericSetCoreOverL} is $\leq\kappa$-closed,
  $\nu$ continues to be a normal measure on $\kappa$ in
  $V[G][H]$. Since $L[\mu]\subseteq L[S^{V[G][H]}]$ by
  Theorem~\ref{thm:LmuSC}, in $L[S^{V[G][H]}]$, we can define $\nu^*$
  such that $A\in \nu^*$ if and only if there is $\bar A\in \mu$ with
  $\bar A\subseteq A$, and clearly $\nu^*$ must be a normal measure on
  $\kappa$ in $L[S^{V[G][H]}]$.

  The argument for $\leq\kappa$-closed $\p$ is even easier because
  $\mu$ remains a normal measure on $\kappa$ in $V[G][H]$.
\end{proof}
Moreover, we get the following variant of Theorem~\ref{th:GCHfails} in
the presence of a measurable cardinal.
\begin{thm}
It is consistent that the stable core has a measurable cardinal above
which the $\GCH$ fails at all regular cardinals.
\end{thm}

\section{Measurable cardinals are not downward absolute to the Stable Core}\label{sec:downwardAbsoluteness}
In this section, we show that it is consistent that measurable cardinals are not downward absolute to the stable core. We will use a modification of Kunen's classical argument that weakly compact cardinals are not downward absolute \cite{kunen:saturatedIdeals}.
\begin{thm}
Measurable cardinals are not downward absolute to the stable core. Indeed, it is possible to have a measurable cardinal in $V$ which is not even weakly compact in the stable core.
\end{thm}
\begin{proof}
  Suppose $V=L[\mu]$, where $\mu$ is a normal measure on a measurable
  cardinal $\kappa$.

  Let $\p_\kappa$ be the Easton support iteration of length $\kappa$
  forcing with $\Add(\alpha,1)$ at every stage $\alpha$ such that
  $\alpha$ is a regular cardinal in $V^{\p_\alpha}$. It is a standard
  fact that whenever $\GCH$ holds, which is the case in $V=L[\mu]$,
  $\p_\kappa$ preserves all cardinals, cofinalities, and the $\GCH$
  (see, for example, \cite{cummings:handbook}). Let
  $G\subseteq \p_\kappa$ be $V$-generic. In $V[G]$, let $\q$ be the
  forcing to add a homogeneous $\kappa$-Souslin tree and let $T$ be
  the $V[G]$-generic tree added by $\q$ (the forcing originally
  appeared in \cite{kunen:saturatedIdeals} and a more modern version
  is written up in \cite{gitman:welch}). In $V[G][T]$, let $\mathbb C$
  be the forcing, as in the proof of
  Theorem~\ref{th:addGenericSetCoreOverL} for the case $m=1$, to code $T$, using the strong limit cardinals, into the stable
  core of a forcing extension high above $\kappa$ and let
  $H\subseteq\mathbb C$ be $V[G][T]$-generic. Finally, in
  $V[G][T][H]$, we force with $T$ to add a branch to it, and let
  $b\subseteq T$ be a $V[G][T][H]$-generic branch of $T$. Note that,
  as a notion of forcing, $T$ is $\lt\kappa$-distributive and
  $\kappa$-cc. As Kunen showed, the combined forcing
  $\q*\dot T$, where $\dot T$ is the canonical name for $T$, has a
  $\lt\kappa$-closed dense subset, and therefore is forcing equivalent
  to $\Add(\kappa,1)$ \cite{kunen:saturatedIdeals}. Thus, the
  iteration $\p_\kappa*\dot \q*\dot T$ is forcing equivalent to
  $\p_\kappa*\Add(\kappa,1)$.

Obviously, $\kappa$ is not even weakly compact in $V[G][T]$, and hence also not in $V[G][T][H]$ because the coding forcing $\mathbb C$ is highly closed and so cannot add a branch to $T$. Thus, $\kappa$ is also not weakly compact in the stable core of $V[G][T][H]$ because, by our coding, $L[S^{V[G][T][H]}]$ has the $\kappa$-Souslin tree $T$.

Next, let's argue that the measurability of $\kappa$ is resurrected in the final model $V[G][T][H][b]$. Standard lifting arguments show that $\kappa$ is measurable in $V[G][T][b]$, which is a forcing extension by $\p_\kappa*\Add(\kappa,1)$. But $V[G][T][H][b]$ and $V[G][T][b]$ have the same subsets of $\kappa$, which means that $\kappa$ is also measurable in $V[G][T][H][b]$.

For ease of reference, let $W=V[G][T][H]$. We will argue that the stable core of $W[b]$ is the same as the stable core of $W$, so that $\kappa$ is not weakly compact there.

Note, using Claim~\ref{cl:codingPreservesGCH} in the proof of Theorem~\ref{th:addGenericSetCoreOverL}, that all forcing extensions in this argument satisfy the $\GCH$. Observe now that since $W$ and $W[b]$ have the same cardinals (since forcing with $T$ is $\lt\kappa$-distributive and $\kappa$-cc), and the $\GCH$ holds in both models, they have the same strong limit cardinals (namely the limit cardinals). Thus, we have that $S_1^W=S_1^{W[b]}$. The remaining arguments will therefore assume that $n\geq 2$ in the triple $(n,\alpha,\beta)$.

It is easy to see that for $\alpha<\beta\leq\kappa$, a triple $(\alpha,\beta,n)\in S^W$ if and only if it is in $S^{W[b]}$ because forcing with the tree $T$ does not add small subsets to $\kappa$.

The case $\kappa<\alpha<\beta$ follows by Proposition~\ref{prop:liftingElementarity} because $\alpha$ is above the size of the forcing $T$.

Finally, we consider the remaining case $\alpha\leq\kappa<\beta$.  Observe that for every strong limit cardinal $\alpha<\kappa$, $H_\alpha^W$  satisfies the assertion that for every successor cardinal $\gamma^+$, there is an $L_\alpha[H_{\gamma^+}]$-generic filter for $\Add(\gamma^+,1)^{L_\alpha[H_{\gamma^+}]}$. The reason is that $H_{\gamma^+}^{V[G]}=H_{\gamma^+}^{V[G_\gamma]}$, where we factor $\p_\kappa\cong \p_\gamma*\p_\tail$ and correspondingly factor $G\cong G_\gamma*G_\tail$. The complexity of the assertion is $\Pi_2$ because we can express it as follows: $$\forall \bar\gamma\forall H\exists \gamma<\bar{\gamma}\, [(\gamma\text{ is regular and }\gamma^{++}=\bar\gamma\text{ and }H=H_{\gamma^+})\rightarrow $$$$\exists g\exists Y\, (Y=L_{\bar\gamma}[H]\text{ and }g\text{ is $Y$-generic for }\Add(\gamma^+,1)^Y].$$
However, $H_\beta^W$ cannot satisfy this assertion because it obviously cannot have an $L_\beta[H_{\gamma^+}]$-generic filter for $\Add(\gamma^+,1)^{L_\beta[H_{\gamma^+}]}$, where $\gamma=\kappa^{++}$, because the $H_{\gamma^+}$ of $L_\beta[H_{\gamma^+}]$ is the real $H_{\gamma^+}$ of $W$. The same argument holds for $W[b]$ showing that no triple
$(n,\alpha,\beta)$ can be in either $S^W$ or $S^{W[b]}$ for $n\geq 2$.
\end{proof}

\section{Separating the stable core and $L[\Card]$}\label{sec:Lcard}
Finally, we would like to consider the possible relationships between the stable core $L[S]$ and the model $L[\Card]$.

Even though the stable core can define the class of strong limit cardinals of $V$, there is no reason to believe that it can see the cardinals. Indeed, it is even possible to make $L[\Card]$ larger than the stable core.
\begin{thm}
It is consistent that $L[S]\subsetneq L[\Card]$.
\end{thm}
\begin{proof}
  We start in $L$. Force to add an $L$-generic Cohen real $r$. Next,
  force with the full support product $\p=\Pi_{k<\omega}\q_k$, where
  if $k\in r$, then $\q_k=\Coll(\aleph_{2k},\aleph_{2k+1})$, and
  otherwise $\q_k$ is the trivial forcing. The forcing $\p$ codes $r$
  into the cardinals of the forcing extension. Suppose
  $H\subseteq \Pi_{k<\omega}\q_k$ is $L[r]$-generic, and observe that
  $r\in L[\Card^{L[r][H]}]$ because it can be constructed by comparing
  the cardinals of $L$ with the cardinals of $L[r][H]$. However, the
  stable core $L[S^{L[r][H]}]=L$ remains unchanged because we
  preserved the strong limit cardinals, and for the slices $S_n$ of
  the stable core with $n\geq 2$, it suffices that the forcing has
  size smaller than the second strong limit cardinal.
\end{proof}
Next, let's show that in various situations, $L[\Card]$ can be a proper submodel of the stable core $L[S]$.
\begin{thm}\label{th:LcardNotStable}
It is consistent that $L[\Card]\subsetneq L[S]$.
\end{thm}
\begin{proof}
  We start in $L$ and force to add a Cohen real $r$. We then code $r$
  into the stable core of a further forcing extension using the coding
  forcing from the proof of Theorem~\ref{th:GCHfails}. More precisely,
  we let $\delta_0$ be any sufficiently large singular
  strong limit cardinal, and let
  $\la (\beta_n,\beta_n^*)\mid n<\omega\ra$ be the sequence of
  $\omega$-many pairs of successive strong limit cardinals above
  $\delta_0$ (note that they must all be singular), which will be our
  coding pairs. Now define that $\mathbb C_n$ is trivial for
  $n\not\in r$, and otherwise let
  $\mathbb C_n=\Add(\delta_n^+,\beta_n)$, where
  $\delta_n=\beta_{n-1}^*$ for $n>0$. By the arguments given in the
  proof of Theorem~\ref{th:GCHfails}, the full support forcing
  $\mathbb C=\Pi_{n<\omega}\mathbb C_n$ is cardinal preserving. Let
  $H\subseteq \mathbb C$ be $L[r]$-generic. Now observe that since we
  have $\Card^L=\Card^{L[r][H]}$, it follows that
  $L[\Card^{L[r][H]}]=L$, but $L[r]\subseteq L[S^{L[r][H]}]$.
\end{proof}

\begin{thm}
It is consistent that $L[S]$ has a measurable cardinal and $L[\Card]\subsetneq L[S]$.
\end{thm}
\begin{proof}
  We start in a model $V=L[\mu]$ with a measurable cardinal $\kappa$
  and force to add a Cohen subset to some $\delta\gg\kappa$. Let
  $G\subseteq \Add(\delta,1)$ be $V$-generic. We then code $G$ into
  the stable core using the cardinal preserving coding forcing $\mathbb C$ from the proof of Theorem~\ref{th:GCHfails}. Let $H$ be $V[G]$-generic for
  $\mathbb C$. Since we forced high above $\kappa$, $\kappa$ remains
  measurable in $L[S^{V[G][H]}]$ as in Theorem
  \ref{thm:codingintoLmu}. Because $\mathbb C$ preserves cardinals,
  $L[\Card^{V[G][H]}]=L[\Card^V]=V$, but $L[S^{V[G][H]}]$ contains $G$
  by construction.
\end{proof}

We also separate the models $L[\Card]$ and $L[S]$ by showing that, for
each $n < \omega$, if $m_{n+1}^{\#}$ exists, then $m_n^{\#}\in
L[S]$. Recall that even $m_1^{\#}$ cannot be an element of $L[\Card]$
(Corollary~\ref{cor:mouseNotInL[C]}).

Given a class $A$, we will say that a cardinal $\gamma$ is \emph{$\Sigma_n$-stable relative to $A$} if \hbox{$\la H_\gamma,\in,A\ra\prec_{\Sigma_n} \la V,\in,A\ra$}. We will say that a class $B$ is \emph{$\Sigma_n$-stable relative to $A$} if every $\gamma\in B$ is $\Sigma_n$-stable relative to $A$. Let us say that a cardinal is \emph{strictly $n$-measurable} if it is $n$-measurable, but not $n+1$-measurable.

\begin{thm}\label{th:L[C1,C2]}
Suppose that $C_1\supseteq C_2$ are class clubs of uncountable cardinals such that $C_2$ is $\Sigma_1$-stable relative to $C_1$. Then there is an
$\Ord$-length iteration of the mouse $m_2^\#$ such that in the direct limit model $M_{C_1,C_2}$ (truncated at $\Ord$) the strictly $1$-measurable cardinals are precisely the elements of $\hat C_1$ and the $2$-measurable cardinals are precisely the elements of $\hat C_2$.
\end{thm}
\begin{proof}
Let $C_1=\la \alpha_\xi\mid\xi\in\Ord\ra$ and let $\la \gamma_\xi\mid\xi\in\Ord\ra$ be a sequence such that $\alpha_{\gamma_\xi}$ is the $\xi$-th element of $C_2$ in the enumeration of $C_2$. The iteration will closely resemble the iteration from the proof of Theorem~\ref{thm:L[C]}.

Iterate the first measurable cardinal $\kappa_0$ of $m_2^{\#}$
$\alpha_0$-many times, so that $\kappa_0$ iterates to
$\alpha_0$, and let $M_{\alpha_0}$ be the iterate. Continue to iterate
measurable cardinals onto elements of $\hat C_1$ until we reach for the
first time a direct limit stage $\eta_0$ where in the model
$M_{\eta_0}$ all measurable cardinals below the first 2-measurable
cardinal are elements of $\hat C_1$. It is not difficult to see that
$\eta_0$ is the first cardinal such that $\eta_0=\alpha_{\eta_0}$, the
$\eta_0$-th element of $C_1$, and that in $M_{\eta_0}$, $\eta_0$
is the first 2-measurable cardinal. Since $C_2$ is $\Sigma_1$-stable
relative to $C_1$, $\eta_0$ must be below the first element of
$C_2$. To achieve our goal of making the least 2-measurable the least
element of $C_2$, at this stage, we iterate up $\eta_0$ to obtain a
model $M_{\eta_0+1}$ with more strictly 1-measurable cardinals and continue
iterating measurable cardinals onto elements of $\hat C_1$. Let
$\eta_\xi$ be the $\xi$-th stage where we iterate up the first
2-measurable cardinal $\eta_\xi$ as above. Since
$\alpha_{\gamma_0}$, the least element of $C_2$, is $\Sigma_1$-stable
relative to $C_1$, it must be that $\eta_\xi<\alpha_{\gamma_0}$ for
every $\xi<\alpha_{\gamma_0}$ as every iteration of a shorter length
$\beta<\alpha_{\gamma_0}$ (of the kind we have been doing) has to be
an element of $H_{\alpha_{\gamma_0}}$ by $\Sigma_1$-elementarity. We
would like to argue that the thread $t$ in the stage
$\alpha_{\gamma_0}$ direct limit such that $t(\xi)$ is the first
2-measurable cardinal maps to $\alpha_{\gamma_0}$ as desired. It
suffices to observe that every ordinal thread $s$ below $t$ must be
constant from some stage onward. So suppose that $s<t$, which by
definition of direct limit means that on a tail of stages $\xi$,
$s(\xi)<t(\xi)$. Fix some such $\xi$ in the tail and consider a stage
$\eta_{\bar\xi}>\xi$ where we have
$s(\eta_{\bar \xi})<t(\eta_{\bar\xi})=\eta_{\bar\xi}$. Here the
equality holds by elementarity as $\eta_{\bar\xi}$ is the first
$2$-measurable cardinal in $M_{\eta_{\bar\xi}}$. Since the critical
points of the iteration after this stage are above $\eta_{\bar\xi}$,
the thread $s$ remains constant from that point onward.

Thus, $\alpha_{\gamma_0}$ must be the first $2$-measurable cardinal in the direct limit model $M_{\alpha_{\gamma_0}}$. Having correctly positioned the first $2$-measurable cardinal, we proceed with iterating the strictly 1-measurable  cardinals onto elements of $\hat C_1$ below the next element of $C_2$. As in the proof of Theorem~\ref{thm:L[C]}, it will be the case that at some limit stages in $C_2$, we will need to use the top measure of $m_2^{\#}$ to create more room for the iteration to proceed.

Let $M$ be the resulting model obtained as the direct limit along the iteration embeddings and let $M_{C_1,C_2}$ be $M$ truncated at $\Ord$. The construction ensures that the strictly 1-measurable cardinals of $M_{C_1,C_2}$ are precisely the elements of  $\hat C_1$ and $2$-measurable cardinals of $M_{C_1,C_2}$ are precisely the elements of $\hat C_2$.
\end{proof}
Given a club $C$, let $C^*$ denote the club of all limit points of
$C$. Next, let's argue that if $C_1$ and $C_2$ are clubs as above,
then $M_{C_1^*,C_2^*}$ is contained in $L[C_1,C_2]$.
\begin{thm}\label{th:mouseInModel}
  Suppose that $C_1\supseteq C_2$ are class clubs of uncountable
  cardinals such that $C_2$ is $\Sigma_1$-stable relative to
  $C_1$. Then $M_{C_1^*,C_2^*}$ (obtained as in
  Theorem~\ref{th:L[C1,C2]}) is contained in $L[C_1,C_2]$.
\end{thm}
\begin{proof}
Given $\alpha\in\hat C_1^*$, let $U_\alpha\subseteq P(\alpha)^{M_{C_1^*,C_2^*}}$ be the iteration measures in $M_{C_1^*,C_2^*}$, and note that a set from $M_{C_1^*,C_2^*}$  is in $U_\alpha$ if and only if it contains a tail of $C_1\cap \alpha$. Let $\vec U=\la U_\alpha\mid\alpha\in\Ord\ra$. Similarly, for $\beta\in \hat C_2^*$, let $W_\beta\subseteq P(\beta)^{M_{C_1^*,C_2^*}}$ be the iteration measures in $M_{C_1^*,C_2^*}$, and let $\vec W=\la W_\beta\mid\beta\in \Ord\ra$. Here we also have that a set is in $W_\beta$ if and only if it contains a tail of $C_2\cap \beta$ because, as we noted in the proof of Theorem~\ref{th:L[C1,C2]}, at every stage in $C_2$ we iterate up the measure on the $2$-measurable cardinal until we reach an element of $\hat C_2^*$. Finally, observe that $M_{C_1^*,C_2^*}=L[\vec U,\vec W]$, and thus, it is contained in $L[C_1,C_2]$.

\end{proof}

Theorems~\ref{th:L[C1,C2]}~and~\ref{th:mouseInModel} easily generalize to $n$ nested clubs $C_1,\ldots, C_n$ such that $C_i$ is $\Sigma_1$-stable relative to $C_1,\ldots,C_{i-1}$ for all $1<i\leq n$.
\begin{thm}\label{th:L[C1,...,Cn]}
Suppose that $C_1\supseteq C_2\supseteq \cdots\supseteq C_n$ are class clubs of uncountable cardinals such that $C_i$ is $\Sigma_1$-stable relative to $C_1,\ldots, C_{i-1}$ for all $1< i\leq n$. Then there is an $\Ord$-length iteration of the mouse $m_n^{\#}$ such that in the direct limit model $M_{C_1,\ldots,C_n}$ (truncated at $\Ord$), for $1\leq i\leq n$, the strictly $i$-measurable cardinals are precisely the elements of $\hat C_i$.
\end{thm}
\begin{thm}\label{th:mouseInModelGeneral}
  Suppose that $C_1\supseteq C_2\supseteq \cdots\supseteq C_n$ are
  class clubs of uncountable cardinals such that $C_i$ is
  $\Sigma_1$-stable relative to $C_1,\ldots, C_{i-1}$ for all
  $1<i\leq n$. Then $M_{ C_1^*,\ldots,C_n^*}$ (obtained as in
  Theorem~\ref{th:L[C1,...,Cn]}) is contained in $L[C_1,\ldots, C_n]$.
\end{thm}
\begin{thm}
  For all $n<\omega$, if $m_{n+1}^\#$ exists, then $m_n^\#$ is in the
  stable core.
\end{thm}
\begin{proof}
By Proposition~\ref{prop:SigmaElementaryClubs}, for $i\geq 1$, the stable core can define class clubs $C_i$ of strong limit cardinals $\alpha$ such that $H_\alpha\prec_{\Sigma_i} V$. In particular, each $C_i$ is $\Sigma_1$-stable relative to $C_1,\ldots,C_{i-1}$. Fix some $n<\omega$. If $m_{n+1}^\#$ exists, then by Theorem~\ref{th:mouseInModelGeneral}, $M_{C_1^*,\ldots,C_{n+1}^*}$ is contained in the stable core, and so in particular, the stable core has $m_n^\#$.
\end{proof}

\section{Characterizing models $L[C_1,\ldots,C_n]$}\label{sec:characterization}
In this section, we will generalize Welch's arguments in
\cite{Welch:ClubClassHartigQuantifierModel} to show that, in the
presence of many measurable cardinals, models $L[C_1,\ldots,C_n]$,
where $$C_1\supseteq C_2\supseteq\cdots\supseteq C_n$$ are class clubs
of uncountable cardinals such that $C_i$ is $\Sigma_1$-stable relative
to\break $C_1,\ldots,C_{i-1}$, are truncations to $\Ord$ of forcing extensions of an iterate of the mouse $m_n^\#$ via a full support product of Prikry forcings.

In \cite{Fuchs:DiscreteProductPrikry}, Fuchs defined, given a discrete
set $D$ of measurable cardinals, a Prikry-type forcing $\p_D$ to
singularize them all, as follows. For every $\alpha\in D$, we fix a
normal measure $\mu_\alpha$ on $\alpha$ with respect to which the
forcing $\p_D$ will be defined. Conditions in $\p_D$ are pairs
$\la h, H\ra$ such that $H$ is a function on $D$ with
$H(\alpha)\in \mu_\alpha$ and $h$ is a function on $D$ with
finite support such that $h(\alpha)\in [\alpha]^{\lt\omega}$ is a
finite sequence of elements of $\alpha$ below the least element of
$H(\alpha)$ and above all $\beta\in D$ with $\beta<\alpha$. Extension is defined by $(h,H)\leq (f,F)$ if for all $\alpha\in D$,
$H(\alpha)\subseteq F(\alpha)$, $h(\alpha)$ end-extends
$f(\alpha)$, and
$h(\alpha)\setminus f(\alpha)\subseteq F(\alpha)$. Note that the first
coordinate of the pair has finite support while the second coordinate
has full support so that the forcing is a mix of a finite support and
a full support-product. It is not difficult to see that the Magidor
iteration of Prikry forcing for a discrete sequence $D$ of measurable
cardinals is equivalent to $\p_D$. For the definition and properties of the Magidor iteration see Section 6 of \cite{gitik:handbook}.
\begin{thm}[see \cite{Fuchs:DiscreteProductPrikry}]\label{thm:propertiesOfPrikry}
The forcing $\p_D$ has the $|D|^+$-cc, preserves all cardinals, and preserves all cofinalities not in $D$.
\end{thm}
\noindent The forcing $\p_D$ also has the Prikry property, namely, given a condition $(h,H)\in\p_D$ and a sentence $\varphi$ of the forcing language, there is a condition $(h,H^*)$ deciding $\varphi$ such that for every $\alpha\in D$, $H^*(\alpha)\subseteq H(\alpha)$ (see \cite{gitik:handbook} Section~6 for details).

For an ordinal $\lambda$, let $D_{<\!\lambda}$ be $D\restrict\lambda$ and $D_{\geq\!\lambda}$ be the rest of $D$. The forcing $\p_D$ factors as $\p_{D_{<\!\lambda}}\times \p_{D_{\geq\!\lambda}}$.
\begin{prop}\label{prop:prikryProperty}
Suppose $f$ is $V$-generic for $\p_D$. If $g:\gamma\to V_\gamma$ is a function in $V[f]$ with $\gamma<\lambda$, then $g$ is added by $f\restrict\lambda$.
\end{prop}
\begin{proof}
It suffices to see that $\p_{D_{\geq\!\lambda}}$ cannot add $g$ over $V[f\restrict\lambda]$. Let $\dot g$ be a $\p_{D_{\geq\!\lambda}}$-name for $g$ so that $\one_{\p_{D_{\geq\!\lambda}}}\forces \dot g:\check\gamma\to \check{V}_\gamma$. By the Prikry property of $\p_{D_{\geq\!\lambda}}$, for every $x\in V_\gamma$ and $\alpha<\gamma$, there is some condition $(\emptyset, H_{x,\alpha})$ deciding whether $\dot g(\check\alpha)=\check x$. There are less than $\kappa$ many such conditions, where $\kappa$ is the least measurable cardinal in $D$ greater than or equal to $\lambda$. Thus, we can intersect all the measure one sets on each coordinate of $H_{x,\alpha}$ to obtain a condition $(\emptyset, H)$ below all of the $(\emptyset, H_{x,\alpha})$. Clearly $(\emptyset, H)$ decides $\dot g$.
\end{proof}

Fuchs showed that the forcing $\p_D$ has a Mathias-like criterion for establishing when a collection of sequences is generic for it.
\begin{thm}[\cite{Fuchs:DiscreteProductPrikry}]\label{th:MathiasCriterion}
A function $f$ on $D$ such that $f(\alpha)\in[\alpha]^\omega$ is an $\omega$-sequence in $\alpha$ above $\beta\in D$ for every $\beta<\alpha$ is generic for $\p_D$ if and only if for every function $H$ on $D$ with $H(\alpha)\in \mu_\alpha$, $\bigcup_{\alpha\in D}f(\alpha)\setminus H(\alpha)$ is finite.
\end{thm}
We will now give the technical set-up for the forcing construction that we want to perform over iterates of the mice $m_n^\#$.

Let $\ZFC^-_I$ be the theory consisting of $\ZFC^-$ together with the assertion that there is a largest cardinal $\kappa$ and that it is inaccessible, namely $\kappa$ is regular and for every $\alpha<\kappa$, $P(\alpha)$ exists and has size smaller than $\kappa$. Note that, in particular, $V_\alpha$ exists in models of $\ZFC^-_I$ for all ordinals $\alpha\leq\kappa$. Natural models of $\ZFC^-_I$ are $H_{\kappa^+}$ for an inaccessible cardinal $\kappa$. The theory $\ZFC^-_I$ is bi-interpretable with the second-order class set theory ${\rm KM}+{\rm CC}$, Kelley-Morse set theory (${\rm KM}$) together with the Class Choice Principle $({\rm CC})$ \cite{Marek:KM}. Models of Kelley-Morse are two-sorted of the form $\mathscr V=(V,\in,\mathcal C)$, with $V$ consisting of the sets, $\mathcal C$ consisting of classes, and $\in$ being a membership relation between sets as well as between sets and classes. The axioms of Kelley-Morse are $\ZFC$ together with the following axioms for classes: extensionality, existence of a global well-order class, class replacement asserting that every class function restricted to a set is a set, and comprehension for all second-order assertions. The Class Choice Principle ${\rm CC}$ is a scheme of assertions, which asserts for every second-order formula $\varphi(x,X,Y)$ that if for every set $x$, there is a class $X$ such that $\varphi(x,X,Y)$ holds, then there is a single class $Z$ choosing witnesses for every set $x$, in the sense that $\varphi(x,Z_x,Y)$ holds for every set $x$, where $Z_x=\{y\mid \la x,y\ra\in Z\}$ is the $x$-th slice of $Z$. If $\V=(V,\in,\mathcal C)$ is a model of ${\rm KM}+{\rm CC}$, then the collection of all extensional well-founded relations in $\mathcal C$, modulo isomorphism and with a natural membership relation, forms a model $M_{\V}$ of $\ZFC^-_I$, whose largest cardinal $\kappa$ is (isomorphic to) $\Ord$, such that $V_{\kappa}^{M_{\V}}\cong V$ and the collection of all subsets of $V_{\kappa}^{M_{\V}}$ in $M_{\V}$ is precisely $\mathcal C$ (modulo the isomorphism).\footnote{We will from now on ignore the isomorphism and assume we have actual equality.} On the other hand, given any model $M\models\ZFC^-_I$, we have that $\mathscr V=(V_\kappa^M,\in,\mathcal C)$, where $\mathcal C$ consists of all subsets of $V_\kappa^M$ in $M$, is a model of ${\rm KM}+{\rm CC}$, and moreover, $M_{\V}$ is then precisely $M$.

The bi-interpretability of the two theories was used by Antos and Friedman in \cite{AntosFriedman:HyperclassForcing} to develop a theory of hyperclass forcing over models of ${\rm KM}+{\rm CC}$. A \emph{hyperclass forcing} over a model $\V=(V,\in,\mathcal C)\models{\rm KM}+{\rm CC}$ is a partial order on a sub-collection of $\mathcal C$ that is definable over $\V$. Suppose that $G$ is $\p$-generic for a hyperclass-forcing $\p$ over $\V$, meaning that it meets all the definable dense sub-collections of $\p$. We move to $M_{\V}$, over which $\p$ is a definable class-forcing, and consider the forcing extension $M_{\V}[G]$. The forcing $\p$ may not preserve $\ZFC^-_I$, but whenever it does, we define that the hyperclass-forcing extension $\V[G]$ is the Kelley-Morse model consisting of $V_\kappa^{M_{\V}[G]}$ together with all the subsets of $V_\kappa^{M_{\V}[G]}$ in $M_{\V}[G]$.

An $\Ord$-length iterate $M$ of a mouse $m^{\#}_n$ is obviously a model of $\ZFC^-_I$ with largest cardinal $\Ord$, and moreover it has a definable global well-ordering. Thus, $M$ naturally gives rise to a model of ${\rm KM}+{\rm CC}$, namely its truncation at $\Ord$, whose classes are the subsets of $V^M_{\Ord}$.

Let $M$ be a model of $\ZFC^-_I$ with a largest cardinal $\kappa$ and a definable well-ordering of the universe. Let $D$ be a discrete set in $M$ of measurable cardinals below $\kappa$ and suppose that $D$ is unbounded in
$\kappa$. Over $M$, $\p_D$ is an $\Ord$-cc class forcing notion (with set-sized
antichains). Class forcing works the same way over models of $\ZFC^-$ as it does over models of $\ZFC$. Pretame forcing (see \cite{friedman:classforcing} for definition and properties) preserves $\ZFC^-$ to forcing extensions and has definable forcing relations (this is due to Stanley and can be found in \cite{PeterHolyRegulaKrapfPhilippSchlicht:classforcing2}), and $\Ord$-cc forcing is pretame. Although ``mixing of names" is not always doable with class forcing that has proper class-sized antichains, it still works for $\Ord$-cc forcing. Finally, the existence of a definable global well-ordering gives that the Mathias criterion of Theorem~\ref{th:MathiasCriterion} still holds in this setting.

Forcing with $\p_D$ preserves the inaccessibility of $\kappa$ by Theorem~\ref{thm:propertiesOfPrikry}, while singularizing all the measurable cardinals below it. Thus, in particular, a forcing extension by $\p_D$ remains a model of $\ZFC^-_I$.
\begin{prop} Given a $\p_D$-generic $f$, we have that $V_\kappa^{M[f]}=V_\kappa^M[f]$.
\end{prop}
\begin{proof}
The one inclusion $V_\kappa^M[f]\subseteq V_\kappa^{M[f]}$ is clear. For the other inclusion suppose that $a\in V_\kappa^{M[f]}$. There is some $\beth$-fixed point cardinal $\alpha<\kappa$ in $M[f]$ such that $a\in V_\alpha^{M[f]}$, so that $a$ is coded there by a subset $A$ of $\alpha$ in $M[f]$. By Proposition~\ref{prop:prikryProperty}, $A$ must be added by some initial segment of the product $\p_D$, and therefore is an interpretation of a name in $V_\kappa^M$ by an initial segment of $f$.
\end{proof}
We can also view $\p_D$ as a hyperclass forcing over the model $\V=(V_\kappa^M,\in,\mathcal C)$, with $\mathcal C$ being the collection of all subsets of $V_\kappa^M$ in $M$. Since $M[f]$ is a model of $\ZFC^-_I$ (because $\kappa$ remains inaccessible), we can form the forcing extension $\V[f]$, and it is (by definition) the model $(V_\kappa[f],\in,\mathcal C^*)$ with $\mathcal C^*$ being the collection of all subsets of $V_\kappa^M[f]$ in $M[f]$.

Now we go back to our specific setting in which we consider $\Ord$-length iterates of the mice $m^{\#}_n$.

Let $C$ be the class club of limit cardinals and let $M$ be the
(non-truncated) iterate model of $m_1^{\#}$ constructed by Welch (see Section~\ref{sec:measurables}). The model $M$ satisfies $\ZFC^-_I$ with the largest cardinal $\Ord$ and has a definable well-ordering of the universe. Let $D=\hat C$ be the collection of all measurable cardinals of $M$. Let $f$ be the function on $D$ such that $f(\omega\cdot\alpha+\omega)=\la \omega\cdot\alpha+n\mid n<\omega\ra$. Welch showed that $f$ is generic for $\p_D$ (defined using measures arising in the iteration) by verifying the Mathias criterion. As before, let $M_C$ be $M$ truncated at $\Ord$. Let $\mathcal C$ be the collection of subsets of $M_C$ in $M$ and $\mathcal C^*$ be the collection of all subsets of $M_C[f]$ in $M[f]$. With this set-up, Welch proved the following theorem.
\begin{thm}[Welch \cite{Welch:ClubClassHartigQuantifierModel}]\label{th:hyperclassWelch}
The model $M[\Card]$ is a class forcing extension of $M$ by the class forcing $\p_D$. Equivalently, the second-order model $(M_C[\Card], \in,\mathcal C^*)$ is a hyperclass-forcing extension of $(M, \in,\mathcal C)$ by the hyperclass forcing $\p_D$.
\end{thm}
Indeed, it is not difficult to see that the function $f$ cannot be added by any set forcing over $M$ (or equivalently, cannot be added by any class forcing over $M_C$).
\begin{thm} \label{th:OnlyByClassForcing} In the notation of Theorem~\ref{th:hyperclassWelch}, $f$ is not set-generic over $M$. Equivalently, $f$
is not class-generic over the second-order model $(M_C[f],\in, \mathcal C)$.
\end{thm}
\begin{proof}
Consider the regressive function $g$ with domain $D$ defined by $g(\alpha) =
\text{min}(f(\alpha))$. If $g_0$ is any regressive function in $M$ on $D$ then by
genericity, $g$ dominates $g_0$ at all but finitely many elements of $D$. But if $\p$ is
any set-forcing of $M$, $\p$ has size at most $\Ord$ in $M$, and therefore we will argue that $\p$ cannot add such a
dominating function. Let $\{p_\xi\mid\xi\in\Ord\}$ be a listing of the elements of $\p$ in which every element of $\p$ appears cofinally often. Let $\dot g$ be a $\p$-name for a regressive function on $D$. For every $\xi \in \Ord$, choose a condition
$p^*_\xi$ extending $p_\xi$ that decides $\dot g(\alpha_\xi)=\beta_\xi$,
where $\alpha_\xi$ is the $\xi$-th element of $D$. Define $g_0(\alpha_\xi) = \beta_\xi + 1$. Then
any condition $p_\xi$ in $\p$ has an extension forcing $g_0(\alpha_\xi) > \dot g(\alpha_\xi)$.
Since every $p\in\p$ appears in our listing cofinally often, for every $\xi < \kappa$, every
$p\in\p$ has an extension forcing $g_0(\alpha_\xi) > \dot g(\alpha_\xi)$, which means that
$\dot g$ cannot be forced to dominate all regressive $g_0$ on $D$ in $M$ on a final
segment of $D$.

\end{proof}

We will need to make use of the following theorem.
\begin{thm}\label{th:PrikryForcingPreservesExtraMeasurables}
Forcing with $\p_D$ preserves all measurable cardinals not in $D$. Indeed, if $\kappa\notin D$ is a measurable cardinal in $M$ and $\mu$ is a normal measure on $\kappa$ that does not concentrate on measurable cardinals, then $\kappa$ has a normal measure $\bar\mu\in M[f]$, a forcing extension by $\p_D$, such that $\bar A\in\bar\mu$ if and only if there is $A\in \mu$ with $A\subseteq\bar A$.
\end{thm}
\noindent See Section~6 of \cite{gitik:handbook} for a proof.

Suppose that $m^{\#}_2$ exists and $C_1\supseteq C_2$ are class clubs of uncountable cardinals such that $C_2$ is $\Sigma_1$-stable relative
  to $C_1$. Let $M$ be the (untruncated) iterate of $m^{\#}_2$ constructed as in the proof of Theorem~\ref{th:L[C1,C2]} for the clubs $C_1^*$ and $C_2^*$ consisting of the limit points of $C_1$ and $C_2$ respectively. With this set-up, we get the following generalization of Welch's theorem.
\begin{thm}
The model $M[C_1,C_2]$ is a forcing extension of $M$ by the class forcing iteration $\p_{D_2}*\dot\p_{D_1}$, where $D_i$, for $i=1,2$, is the class of strictly $i$-measurable cardinals. The iteration $\p_{D_2}*\dot\p_{D_1}$ is equivalent to the product $\p_{D_2}\times \p_{D_1}$. Moreover, $M_{C_1^*,C_2^*}[C_1,C_2]=L[C_1,C_2]$, and the latter is then the first-order part of a hyperclass-forcing extension of the resulting Kelley-Morse model.
\end{thm}

\begin{proof} First, observe that the $2$-step iteration $\p_{D_2}*\dot\p_{D_1}$ makes sense over $M$ because both $D_1$ and $D_2$ are discrete sets of measurable cardinals and after we force with $\p_{D_2}$ and singularize all measurable cardinals in $D_2$, every cardinal in $D_1$ remains measurable by Theorem~\ref{th:PrikryForcingPreservesExtraMeasurables}. For every measurable cardinal $\alpha\in M$, let $\mu_\alpha$ be the normal measure on $\alpha$ arising from the iteration. Since no $\mu_\alpha$ concentrates on measurable cardinals (otherwise $\alpha$ would have Mitchell order 2), by Theorem~\ref{th:PrikryForcingPreservesExtraMeasurables}, every measurable cardinal $\alpha\in D_1$ has in $M[f_2]$ a normal measure $\bar\mu_\alpha$ generated by $\mu_\alpha$. Let $\p_{D_1}$ and $\p_{D_2}$ be defined with respect to the measures $\mu_\alpha$, and let $\bar{\p}_{D_1}=(\dot \p_{D_1})_{f_2}$ be defined with respect to the measures $\bar\mu_\alpha$.

Let $f_1$ be the function on the elements $\alpha$ of $\hat C_1^*$ such that $f_1(\alpha)$ is the $\omega$-sequence of elements of $C_1$ limiting up to $\alpha$. Let $f_2$ be the function on the elements $\alpha$ of $\hat C_2^*$ such that $f_2(\alpha)$ is the $\omega$-sequence of elements of $C_2$ limiting up to $\alpha$. The arguments in \cite{Welch:ClubClassHartigQuantifierModel} already verify that $f_2$ satisfies the Mathias criterion for $\p_{D_2}$ and $f_1$ satisfies the Mathias criterion for $\p_{D_1}$, the Prikry forcing for $D_1$ defined over $M$. Indeed, we will now argue that $\p_{D_1}$ densely embeds into $\bar{\p}_{D_1}$. Since class forcing notions with set-sized antichains which densely embed produce the same forcing extensions (see \cite{PeterHolyRegulaKrapfPhilippSchlicht:classforcing2}), we will be able to assume without loss that we are actually forcing with $\p_{D_1}$.

It suffices to argue that for every function $F$ on $D_1$ in $M[f_2]$ such that $F(\alpha)\in\mu_\alpha$, there is a function $F^*\in M$ such that $F^*(\alpha)\in \mu_\alpha$ and $F^*(\alpha)\subseteq F(\alpha)$ for every $\alpha\in D_1$. Let $\dot F$ be a $\p_{D_2}$-name for $F$ such that $\one_{\p_{D_2}}$ forces that $\dot F(\alpha)\in\check\mu_\alpha$ for every measurable $\alpha$ (this requires mixing names). We will argue in $M$, by induction on $\beta\leq \Ord$, that we can define cohering functions $f_\beta$ on $D_1\restrict \beta$ such that $f_\beta(\alpha)\in\mu_\alpha$ and $\one_{\p_{D_2}}\forces \check f_\beta(\alpha)\subseteq \dot F(\alpha)$ for every $\alpha\in D_1\restrict\beta$. Suppose inductively that we can construct $f_\gamma$ for $\gamma<\beta$ as required. Let's argue that we can construct $f_\beta$. If $\beta$ is a limit ordinal, then $f_\beta$ is just the union of the $f_\gamma$. So suppose that $\beta=\beta^*+1$ and assume that $F$ is defined at $\beta^*$ because otherwise there is nothing to prove. Observe that $F\restrict \beta$ must be added by $\p_{{D_2}_{<\!\beta}}$ by Proposition~\ref{prop:prikryProperty}. Since $\beta^*\in C_1$ cannot be a limit of elements of $C_2$, $\p_{{D_2}_{<\!\beta}}$ must have size $\lambda<\beta^*$. Let $\dot f$ be a $\p_{{D_2}_{<\!\beta}}$-name for $F\restrict\beta$ such that $\one_{\p_{D_2}}\forces \dot f=\dot F\restrict\beta$. For every condition $p\in \p_{{D_2}_{<\!\beta}}$, if $p$ forces that some $A\in \mu_{\beta^*}$ is contained in $\dot f(\beta^*)$, then choose some such $A_p$. Since there are at most $\lambda$-many such sets $A_p\in\mu_{\beta^*}$ and $\lambda<\beta^*$, we can intersect them all to obtain a set $A\in\mu_{\beta^*}$ such that $\one_{\p_{{D_2}_{<\!\beta}}}\forces\check A\subseteq \dot f(\beta^*)$. It follows that $f_\beta$ defined to extend $f_{\beta^*}$ with $f_\beta(\beta^*)=A$ satisfies our requirements.

This completes the argument that $\p_{D_1}$ densely embeds into $\bar{\p}_{D_1}$. The argument also shows that $f_1$ meets the Mathias criterion for $\bar \p_{D_1}$ because it met the Mathias criterion for $\p_{D_1}$ over $M$ and every sequence of measure one sets from $M[f_2]$ can be thinned out on each coordinate to a sequence of measure one sets which exists in $M$. Thus, $f_1$ is $M[f_2]$-generic for $\bar{\p}_{D_1}$.

Finally, by Theorem~\ref{thm:propertiesOfPrikry}, forcing with $\bar{\p}_{D_1}$ preserves the inaccessibility of $\Ord$ in $M[f_2]$, so that we can form the hyperclass forcing extension of the Kelley-Morse model whose first-order part is $M_{C_1^*,C_2^*}$ and the first-order part of the forcing extension is then the model $M_{C_1^*,C_2^*}[f_2][f_1]=M_{C_1^*,C_2^*}[C_1,C_2]$. Using Theorem~\ref{th:mouseInModel}, it is clear that $M_{C_1^*,C_2^*}[C_1,C_2]=L[C_1,C_2]$.
\end{proof}

The characterization easily generalizes to $n$-many clubs $C_1, \dots, C_n$. Suppose that $m_n^{\#}$ exists and $C_1\supseteq C_2\supseteq\ldots\supseteq C_n$ are clubs of uncountable cardinals such that $C_i$ is $\Sigma_1$-stable relative to $C_1,\ldots,C_{i-1}$ for all $1<i\leq n$. Let $M$ be the (untruncated) iterate of $m_n^{\#}$ constructed as usual for the clubs $C_1^*,\ldots,C_n^*$ consisting of the limit points of the clubs $C_1,\ldots,C_n$ respectively.
\begin{thm}\label{th:CharacterizationForNClubs}
The model $M[C_1,\ldots,C_n]$ is a forcing extension of $M$ by the class forcing iteration $\p_{D_n}*\cdots*\dot\p_{D_1}$, where $D_i$, for $1\leq i\leq n$, is the class of strictly $i$-measurable cardinals. The iteration $\p_{D_n}*\cdots*\dot\p_{D_1}$ is equivalent to the product $\p_{D_n}\times\cdots\times \p_{D_1}$. Moreover, $M_{C_1^*,\ldots,C_n^*}[C_1,\ldots,C_n]=L[C_1,\cdots,C_n]$, and the latter is then the first-order part of a hyperclass-forcing extension of the resulting Kelley-Morse model.
\end{thm}
Theorem~\ref{th:OnlyByClassForcing} also generalizes to show that such an extension cannot be obtained by a set forcing over $M$ (or equivalently a class forcing over $M_{C_1^*,\ldots,C_n^*}$).
\section{Open Questions}
The article did not answer several difficult questions about the structure of the stable core. In Sections~\ref{sec:coding}~and~\ref{sec:measurables}, we showed how to code information into the stable core over small canonical inner models using the fact that these models must be contained in the stable core so that we can use them for decoding. In a very recent work Friedman showed that there is a large cardinal notion below a Woodin cardinal such that the stability predicate is definable over an iterate of a mouse with such a large cardinal. This immediately implies that we cannot code any set into the stable core and that there is a bound (below a Woodin cardinal) on the large cardinals that can exist in the stable core \cite{Friedman:LimitationsStableCore}. This leaves the following open questions regarding the structure of the stable core.

We still don't have a precise upper bound on the large cardinals that can exist in the stable core.
\begin{question}
Can the stable core have a measurable limit of measurable cardinals?
\end{question}
For $\HOD$, we know that the $\HOD$ of $\HOD$ can be smaller than $\HOD$  and that any universe $V$ is the $\HOD$ of a class forcing extension of itself.

\begin{question}
Can the stable core of the stable core be smaller than the stable core?
\end{question}

\begin{question}
When is $V$ the stable core of an outer model?
\end{question}
Finally, with regard to Section~\ref{sec:characterization}, we can ask whether the results there generalize to $\omega$-many clubs.
\begin{question}
Is there a version of Theorem~\ref{th:CharacterizationForNClubs} for $\omega$-many clubs?
\end{question}

\bibliographystyle{alpha}
\bibliography{StableCore}
%\nocite{*}

\end{document}